\documentclass{siamltex}%

\usepackage{amsfonts}
\usepackage{amsmath}%
\usepackage{amssymb}%

\usepackage{graphicx}
\usepackage[T1]{fontenc}
\usepackage[utf8]{inputenc}
\usepackage{mathtools} 
\usepackage{multirow}
\usepackage{cleveref}

\usepackage{lipsum}
\usepackage{algorithmic}
\usepackage[ruled,vlined]{algorithm2e}
\usepackage{cite}
\usepackage{epstopdf}
\usepackage{float}
\usepackage{diagbox}
\usepackage{subfigure}
\usepackage{color}
\usepackage[margin=1in]{geometry}
\usepackage{lscape}

\graphicspath{ {./png/} }

\usepackage{epsfig}

\usepackage{graphics}

\usepackage{bm}
\usepackage{latexsym}
\usepackage{cancel}

\usepackage{epic,eepic,euscript,verbatim,afterpage,stmaryrd,paralist}
\usepackage{amscd}
\usepackage{mathrsfs}
\usepackage{fancyhdr,fancybox}
\usepackage{listings}
\usepackage{url} 
\usepackage{enumerate,algorithm2e}
\usepackage{array}
\usepackage{ulem}
\usepackage{appendix}

\usepackage{comment}

\newcommand{\bu}{{\bf u} }

\newcommand{\Dt}{\Delta}
\newcommand{\be}{\begin{equation}}
\newcommand{\ee}{\end{equation}}
\newcommand{\ba}{\begin{array}}
\newcommand{\ea}{\end{array}}
\newcommand{\bea}{\begin{eqnarray}}
\newcommand{\eea}{\end{eqnarray}}
\newcommand{\beas}{\begin{eqnarray*}}
\newcommand{\eeas}{\end{eqnarray*}}
\newcommand{\bx}{{\bf x} }

\newcommand{\divv}{\nabla \cdot}

\newcommand{\cc}{\color{blue}}

\providecommand{\U}[1]{\protect\rule{.1in}{.1in}}

\newtheorem{condition}[theorem]{Condition}

\newtheorem{remark}[theorem]{Remark}

\begin{document}

\title{A Doubly Adaptive Penalty method for the Navier Stokes Equations\thanks{The research was partially supported by NSF grant DMS1817542 and DMS2110379.} }
\author{
Kiera Kean\thanks{Department of Mathematics, Univeristy of Pittsburgh, Pittsburgh, PA 15260 (kkh16@pitt.edu).} \and Xihui Xie\thanks{Department of Mathematics, Univeristy of Pittsburgh, Pittsburgh, PA 15260 (xix55@pitt.edu).} \and Shuxian Xu\thanks{Department of Mathematics, Univeristy of Pittsburgh, Pittsburgh, PA 15260 (shx34@pitt.edu).}}
\maketitle

\begin{abstract}
We develop, analyze and test adaptive penalty parameter methods. 
We prove unconditional stability for velocity when adapting the penalty parameter, $\epsilon,$ and stability of the velocity time derivative under a condition on the change of the penalty parameter, $\epsilon(t_{n+1})-\epsilon(t_n)$. The analysis and tests show that adapting $\epsilon(t_{n+1})$ in response to $\nabla\cdot u(t_n)$ removes the problem of picking $\epsilon$ and yields good approximations for the velocity. 
 We provide error analysis and numerical tests to support these results. We supplement the adaptive-$\epsilon$ method by also adapting the time-step. The penalty parameter $\epsilon$ and time-step are adapted independently.
 We further compare first, second and variable order time-step algorithms. 
Accurate recovery of pressure remains an open problem. 


\end{abstract}


\section{Introduction}

The velocity and pressure of an incompressible, viscous fluid are given by the Navier-Stokes equations. Let $\bu$ denote the fluid velocity, $p$ the pressure, $\nu$ the kinematic viscosity and $f$ an external force:
\begin{align}
    \bu_t-\nu\Delta\bu+\bu\cdot\nabla\bu+\nabla p=f,&\quad
    \nabla\cdot\bu=0,\quad \forall (\bx,t)\in \Omega\times(0,T]. \label{nse1}
\end{align}
The velocity and pressure are coupled together by the incompressibility constraint $\nabla\cdot\bu=0$. Coupled systems require more memory to store and are more expensive to solve. 
Penalty methods and artificial compression methods relax the incompressibility condition and result in a pseudo-compressible system. This allows us to uncouple velocity and pressure, which will reduce storage space and computational complexity.
Penalty methods that allow complete elimination of the pressure variable are the simplest and fastest, and will be studied herein. 

Penalty methods replace $\nabla\cdot\bu=0$ with $\nabla\cdot\bu+\epsilon p=0\,$ where $0<\epsilon<<1$. The pressure can be eliminated using  $\nabla p=-\nabla(1/\epsilon \nabla\cdot\bu)$. As the pressure is entirely eliminated from the system, we do not need to solve for it at every timestep, leading to further increases in speed.  The accuracy of penalty methods is known to be very sensitive to the choice of $\epsilon$ (see Fig. 6.1). This sensitivity suggests considering $\epsilon$ a control and picking $\epsilon$ through a self-adaptive algorithm.
This problem of determining $\epsilon$ self-adaptively is considered herein. 

When adapting the parameter $\epsilon$, $\|\nabla\cdot\bu\|$ is monitored and used to make adjustment to $\epsilon$. 
The stability of the standard penalty method with variable $\epsilon$ is examined in Section \ref{stable-analysis}. No condition on the rate of change of $\epsilon$ is required for the stability of $\|\bu\|$. However, stability of $\|\bu_t\|$ is not unconditional. There is no restriction on the increase of $\epsilon$, however decreasing $\epsilon$ quickly will lead to growth in $\|\bu\|$.
In Section 3.2, we derive condition \eqref{epsilon-restriction}
\begin{equation*}
    (1-k\alpha)\epsilon_n\leq\epsilon_{n+1}\  \text{for some }\alpha>0,
\end{equation*}
 where $k$ is the step-size.
This condition is required for stability of $\|\bu_t\|.$ \Cref{fig:ut_spike} confirms that violating this condition leads to spikes of catastrophic growth in $\|\bu_t\|$. 

The utility of penalty methods lies in accurate velocity approximation at low cost by simple methods. Consistent with this intent, we couple the adaptive $\epsilon$ algorithm with simple, low cost time stepping methods based on the  backward Euler method. Simple time filters allow us to implement an effective variable order, variable time-step adaptive scheme, developing further an algorithm of \cite{guzel2018timefilter}. The self-adaptive $\epsilon$ penalty method can be easily implemented for both constant time-step and variable time-step methods. We develop, analyze and test these new algorithms that independently adapt the time-step $k$ and the penalty parameter $\epsilon$. 

In addition to adapting the time-step, we adapt the order of the method between first and second order.  This variable time-step variable order (VSVO) method performed better than both first and second-order methods in our tests (see \Cref{fig:variable-step} and \Cref{fig:ut_spike}).

.

The rest of the paper is organized as follows. In Section 2, we introduce important notation and preliminary results. In Section 3, stabilities of $\|\bu\|$ and $\|\bu_t\|$ for the variable $\epsilon$ penalty method with constant timestep is presented.
Section 4 presents an error estimate of the semi-discrete, variable $\epsilon$ method.  Using this we develop an effective algorithm which adapts $\epsilon$ and $k$ independently, presented in Section 5. We introduce four different algorithms, including the constant time-step and variable time-step variable $\epsilon$ method. Numerical tests are shown in Section 6 and open problems are presented in Section 7.
\subsection{Review of a Common Penalty Method}
Recall the incompressible Navier-Stokes equations, \eqref{nse1}.
Perturbing the continuity equation by adding a penalty term to the incompressibility condition and explicitly skew-symmetrizing the nonlinear term in the momentum equation in \eqref{nse1} results in the penalty Navier-Stokes equations: 
\begin{align}
    \bu_t-\nu\Delta\bu+\bu\cdot\nabla\bu+\frac{1}{2}(\nabla\cdot\bu)\bu+\nabla p=f,&\quad\label{pnse1}
    \\
    \nabla\cdot\bu+\epsilon p=0. \label{pnse1b}
\end{align}
By  \eqref{pnse1b}, $p = (-1/\epsilon)\nabla \cdot \bu$. Inserting this into \eqref{pnse1} results in a system of $\bu$ only, which is easier to solve than \eqref{nse1}: 
\begin{equation}\label{pnse_eliminate}
    \bu_{\epsilon,t}-\nu\Delta\bu_\epsilon+\bu_\epsilon\cdot\nabla\bu_\epsilon+\frac{1}{2}(\nabla\cdot\bu_\epsilon)\bu_\epsilon-\nabla(\frac{1}{\epsilon}\nabla\cdot\bu_\epsilon)=f.
\end{equation}
From Theorem 1.2 p.120 of Temam \cite{teman1968} we know $\lim_{\epsilon\to 0}(\bu_\epsilon(t),p_\epsilon(t))=(\bu(t),p(t))$. 

Consider the first-order discretization of \eqref{pnse1}-\eqref{pnse1b}. $k_n$ is the n$^{th}$ time-step, $\epsilon_n$ is the parameter $\epsilon$ at n$^{th}$ time-step, $t_0=0, t_n=t_{n-1}+k_n$. Let $\bu^*$ denote the standard (second order) linear extrapolation of $\bu$ at $t_{n+1}$:
\begin{equation*}
    \bu^*=\left(1+\frac{k_{n+1}}{k_n}\right)\bu^n-\frac{k_{n+1}}{k_n}\bu^{n-1}.
\end{equation*}
The backward Euler time discretization gives us 
\begin{align}
\frac{\bu^{n+1}-\bu^n}{k_{n+1}}+\bu^*\cdot\nabla\bu^{n+1}+\frac{1}{2}(\nabla\cdot\bu^*)\bu^{n+1}+\nabla p^{n+1}-\nu\Delta\bu^{n+1} &=f^{n+1},\label{pnseBE} \\
\nabla\cdot\bu^{n+1}+\epsilon_{n+1} p^{n+1} &=0.\label{pneBE2}
\end{align}
As before, we use $ p^{n+1}=\left((-1/\epsilon_{n+1})\nabla\cdot\bu^{n+1}\right)$, to uncouple \eqref{pnseBE}-\eqref{pneBE2} into the following time discrete, velocity only equation
\begin{equation}\label{discrete-stokes-nonlinear}
\frac{\bu^{n+1}-\bu^n}{k_{n+1}}+\bu^*\cdot \nabla \bu^{n+1}+\frac{1}{2}(\nabla\cdot\bu^*)\bu^{n+1} - \nu\Dt \bu^{n+1} - \nabla (\frac{1}{\epsilon_{n+1}}\nabla\cdot \bu^{n+1}) = f^{n+1}.
\end{equation}
For constant $\epsilon_{n+1}=\epsilon, k_{n+1} = k$, \eqref{discrete-stokes-nonlinear} is unconditionally stable by Theorem 4.1 of He and Li \cite{He2010penalty}. The analysis of the stability of the variable $\epsilon$, constant k method is found in Theorem 3.1. Analysis of stability of acceleration $\bu_t$ is found in Theorem 3.3.

\subsection{Related work}
Penalty methods were first introduced by Courant in 1943 \cite{courant1994variational}. They were first applied to the unsteady Navier-Stokes equations by Temam \cite{teman1968}.
Error estimates for continuous time, constant $\epsilon$, \eqref{pnse_eliminate} were proved by Shen in Theorem 4.1 p.395  \cite{Shen1995penalty}.
In Theorem 5.1 p.397, Shen further proved error estimates for the backward Euler time discretization of the penalty Navier-Stokes equations. This analysis suggests a choice of $\epsilon=k$.
Shen \cite{shen1992error} studied higher-order projection schemes in the semi-discrete form and propose a penalty-projection scheme with improved error estimates.
Prohl \cite{prohl1997projection} suggested a new analytical approach to the penalty method.
He \cite{He2005error}, He and Li \cite{He2010penalty} studied fully discrete penalty finite element methods and proved optimal error estimates with conditions on $\epsilon, \Dt t$ and mesh size $h$.

Bercovier and Engelman showed the velocity error of penalty methods is sensitive to the choice of $\epsilon$, see  \cite{BERCOVIER1979181}. If $\epsilon$ is too large, it will poorly model incompressible flow.
Choosing $\epsilon$ too small will cause numerical conditioning problems, see Hughes, Liu and Brooks  \cite{hughes1979finite}. The optimal choice of the penalty parameter also varies depending on the time and space discretization schemes used, see Shen \cite{Shen1995penalty}. \cite{hughes1979finite} introduced a theory for determining the penalty parameter, which only depends on the Reynolds number Re and viscosity $\mu$. 

The penalty method gives inaccurate pressure (see \Cref{constant_error} and \Cref{variable_error}), and we focus on the velocity accuracy in this paper. But pressure recovery is important when calculating quantities based on stresses, e.g. lift and drag coefficients.  The easiest way is by using $\nabla\cdot\bu+\epsilon p=0$ and solve for pressure. There are also other possibilities to recover pressure, e.g. Pressure Poisson equations and momentum equation, see Kean and Schneier \cite{Kean2019error}.

\subsection{Motivation For Choice of Estimator for $\epsilon$}
We choose an estimator to control the residual in the continuity equation, $\|\divv \bu_\epsilon\|$. The immediate choice is to adapt $\epsilon$ based on the size of $\|\divv \bu_\epsilon\|$. However, controlling the relative, not the absolute error is a more logical choice. 
Taking $L^2$ inner product of \eqref{pnse_eliminate} with $\bu_\epsilon$, we get:
\begin{equation*}
    \frac{1}{2}\frac{d}{dt}\|\bu_\epsilon\|^2+\nu\|\nabla\bu_\epsilon\|^2+\frac{1}{\epsilon}\|\nabla\cdot\bu_\epsilon\|^2=(f,\bu_\epsilon).
\end{equation*}
We aim to ensure $\nu\|\nabla\bu_\epsilon\|^2$ does not dominate  $\frac{1}{\epsilon}\|\nabla\cdot\bu_\epsilon\|^2$. This suggests an upper bound for $\epsilon$:

\begin{align*}
    \frac{1}{\epsilon}\|\nabla\cdot\bu_\epsilon\|^2\geq \nu\|\nabla\bu_\epsilon\|^2 
    \quad
    \implies
    \epsilon\leq \frac{1}{\nu}\left(\frac{\|\nabla\cdot\bu_\epsilon\|}{\|\nabla\bu_\epsilon\|}\right)^2.
\end{align*}
This motivates the choice of the estimator to be $\|\divv\bu_\epsilon\|/\|\nabla\bu_\epsilon\|,$ the relative residual. This has the additional benefits of being non-dimensional and independent of the size of $\bu_\epsilon$. Since $\nu$ is constant, scaling by $1/\nu$ is just a change of adaptive tolerance. 
The comparison of absolute and relative residual estimators is presented in Section 6.3.



\section{Notation and preliminaries}
We denote by $\|\cdot\|$ and $(\cdot,\cdot)$ the $L^2(\Omega)$ norm and inner product, respectively. We denote by $\|\cdot\|_{L^p}$ the $L^p(\Omega)$ norm. The velocity space $X$  and pressure space $Q$ are:
\begin{align*}
    &X:=(H_0^1(\Omega))^d,\ \text{where}\ H_0^1(\Omega)=\{v\in L^2(\Omega): \nabla v\in L^2(\Omega)\ \text{and}\ v=0\ \text{on}\ \partial\Omega\},\\
    &Q:=L^2_0(\Omega)=\{q\in L^2(\Omega): \int_\Omega q\ d\bx=0\}.
\end{align*}
Let $X^h\subset X$ be the finite element velocity space and $Q^h\subset Q$ be the finite element pressure space.
We assume that $(X^h, Q^h)$ are conforming and satisfy the following approximation properties and Condition 2.1:
\begin{equation}\label{approx-prop}
\begin{aligned}
    \inf_{v\in X^h}\|u-v\|&\leq Ch^{m+1}|u|_{m+1},\quad u\in H^{m+1}(\Omega)^d,\\
    \inf_{v\in X^h}\|\nabla(u-v)\|&\leq Ch^{m}|u|_{m+1},\quad u\in H^{m+1}(\Omega)^d,\\
    \inf_{q\in Q^h}\|p-q\|&\leq Ch^m|p|_m,\quad p\in H^m(\Omega).
\end{aligned}
\end{equation}
\begin{condition} (The Ladyzhenskaya-Babuska-Brezzi Condition (LBB$^h$) see p.62 \cite{cfdbook}).\\
Suppose $(X^h,Q^h)$ satisfies:
\begin{equation}\label{lbbh}
    \inf_{q^h\in Q^h}\sup_{v^h\in X^h}\frac{(q^h,\nabla\cdot v^h)}{\|\nabla v^h\|\|q^h\|}\geq\beta^h>0,
\end{equation}
where $\beta^h$ is bounded away from zero uniformly in h.

The (LBB$^h$) condition is equivalent to:
\begin{equation*}
    \beta^h\|q^h\|\leq\sup_{v^h\in X^h}\frac{(q^h,\nabla\cdot v^h)}{\|\nabla v^h\|}.
\end{equation*}
\end{condition}
The space $H^{-1}(\Omega)$ denotes the dual space of bounded linear functional defined on $H_0^1(\Omega)$. This space is equipped with the norm:
\begin{equation*}
    \|f\|_{-1}=\sup_{0\neq v\in X}\frac{(f,v)}{\|\nabla v\|}.
\end{equation*}
Let $I^h$ denote the interpolant in the space of $C^0$ piecewise linears, suppose the following interpolation estimate in $H^{-1}(\Omega)$ holds
(see p.160 of \cite{cfdbook})
\begin{equation}\label{negative-norm}
    \|u-I^h(u)\|_{H^{-1}(\Omega)}\leq Ch\| u-I^h(u)\|.
\end{equation}

Denote by $b^*(u,v,w)$, the skew-symmetric trilinear form, is
\begin{equation*}
    b^*(u,v,w):=\frac{1}{2}(u\cdot\nabla v,w)-\frac{1}{2}(u\cdot \nabla w,v)\quad \forall u,v,w\in [H^1(\Omega)]^d.
\end{equation*}
A weak formulation of the penalty NSE is: find $\bu: (0,T]\to X$ such that 
\begin{align*}
    (\bu_t,v)+b^*(\bu,\bu,v)+\nu(\nabla\bu,\nabla v)+\frac{1}{\epsilon}(\nabla\cdot\bu,\nabla\cdot v)&=(f,v),\quad \forall v\in X,\\
    \bu(\bx,0)&=\bu^0(\bx).
\end{align*}
\begin{lemma}(skew-symmetry see p.123 p.155 \cite{cfdbook}, upper bound for the product of three functions see p.11 \cite{cfdbook})\label{skewsymsharper}
There exists $C_1$ and $C_2$ such that for all $u, v,  w\in X$, $b^*(u, v,  w)$ satisfies
\begin{equation*}
    \begin{aligned}
         & b^*(u, v,  w)=(u\cdot \nabla v,  w)+\frac{1}{2}((\nabla\cdot u) v,  w), \\
         & b^*(u, v,  w)\leq C_1 \|\nabla u\| \|\nabla v\| \|\nabla  w\|, \\
         & b^*(u, v,  w)\leq C_2 \sqrt{\|u\|\|\nabla u\|} \|\nabla v\| \|\nabla  w\|. 
    \end{aligned}
\end{equation*}
Moreover, if $ v\in H^2(\Omega)$, then there exists $C_3$ such that
$$
b^*(u, v,  w)\leq C_3 \big({\|u\|} \| v\|_2 \|\nabla  w\|+\|\nabla\cdot u\|\|\nabla v\|\|\nabla  w\|\big).
$$
Further, if $ v\in H^2(\Omega)\cap L^\infty(\Omega)$, then
$$
b^*(u, v,  w)\leq \big(C_3 \| v\|_2 +\| v\|_\infty\big)\|u\|\|\nabla  w\|.
$$
\end{lemma}

\begin{lemma}(The Poincar$\Acute{e}$-Friedrichs' inequality see p.9 \cite{cfdbook})
There is a positive constant $C_{PF}=C_{PF}(\Omega)$ such that
\begin{equation}\label{pf ineq}
    \|\bu\|\leq C_{PF}\|\nabla\bu\|\quad\forall\bu\in X.
\end{equation}
\end{lemma}

\begin{lemma}(A Sobolev inequality see \cite{cfdbook})
Let $\Omega$ be a bounded open set and suppose $\nabla\bu\in L^p(\Omega)$ with $\bu=0$ on a subset of $\partial\Omega$ with positive measure. Then there is a $C=C(\Omega,p)$ such that for $1\leq p< \infty$,
\begin{align*}
    \|\bu\|_{L_{p^\star}}\leq C\|\nabla\bu\|_{L^p}, \\
    \text{where}\ \frac{1}{p^\star}=\frac{1}{p}-\frac{1}{\text{dim}(\Omega)}\;\;\text{if}\;\; p<\text{dim}(\Omega).
\end{align*}
For example, with $p=2$, for $1\leq p^\star < \infty$ in 2d and $1\leq p^\star \leq 6$ in 3d,
\begin{equation}\label{sobolev}
    \|\bu\|_{L_{p^\star}}\leq C\|\nabla\bu\|.
\end{equation}
\end{lemma}

\begin{lemma}(Useful inequalities see p.7 \cite{cfdbook}, polarization identity)
The $L^2$ inner product satisfies the H\"{o}lder's and Young's inequalities: for any $u,v \in L^2(\Omega)$, for any $\delta$, $0 < \delta < \infty$ and $\frac{1}{p}+\frac{1}{q}=1, 1\leq p,q \leq \infty$,
\begin{equation}\label{holder+young}
\begin{aligned}
    (u,v)\leq\|u\|_{L^p}\|v\|_{L^q},\ \text{and}\ 
    (u,v)\leq \frac{\delta}{p}\|u\|_{L^p}^p+\frac{\delta^{-q/p}}{q}\|v\|_{L^q}^q.
\end{aligned}
\end{equation}
Further, for any $u,v,w\in X$, for any $p,q,r, 1\leq p,q,r \leq \infty$, with $\frac{1}{p}+\frac{1}{q}+\frac{1}{r}=1,$
\begin{equation}\label{trilinear ineq}
    \int_\Omega |u||v||w| dx\leq\|u\|_{L^p}\|v\|_{L^q}\|w\|_{L^r}.
\end{equation}
Polarization identity: for any $u,v \in X$
\begin{align}
    (u,v)&=\frac{1}{2}\|u\|^2+\frac{1}{2}\|v\|^2-\frac{1}{2}\|u-v\|^2,\quad\forall\ u,v\in L^2(\Omega). \label{polar}
\end{align}
\end{lemma}

\begin{proposition}(see p.173 of 
\cite{tool2012mathematical}) Let $W^{m,p}(\Omega)$ denote the Sobolev space, let $p\in[1,+\infty]$ and $q\in[p,p^*]$. There is a $C>0$ such that
\begin{equation}\label{sobolev2}
    \|u\|_{L^q}\leq C\|u\|_{L^p}^{1+d/q-d/p}\|u\|_{W^{1,p}}^{d/p-d/q},\quad\forall u\in W^{1,p}(\Omega)
\end{equation}
\end{proposition}
\begin{lemma}(A $L^p-L^2$ type inverse inequality see Lemma 2.1 of Layton \cite{layton1996nonlinear} also similar result of p.112 Theorem of Brenner and Scott \cite{brennerscott2008mathematical})
Let $\theta_0$ be the minimum angle in the triangulation and $M^k=\{v(x):v(x)|_e\in \mathcal{P}_k(e)\ \forall\ e\in\mathcal{T}^h(\Omega)\}$, $\mathcal{P}_k$ being the polynomials of degree $\leq k$. Then, for $\nabla^h$ the elementwise defined gradient operator, there is a $C=C(\theta_0,p,k)$ such that for $2\leq p<\infty, d=2,3$ and all $v\in M^k$,
\begin{equation}\label{lp-l2inverse}
    \|\nabla^h v\|_{L^p(\Omega)}\leq Ch^{\frac{d}{2}(\frac{2-p}{p})}\|\nabla^h v\|.
\end{equation}
\end{lemma}


\begin{proposition}(The continuous inf-sup condition see p.58 \cite{cfdbook})
There is a constant $\beta>0$ such that
\begin{equation}\label{lbb}
    \inf_{q\in Q}\sup_{v \in X} \frac{(q,\nabla\cdot v)}{\|\nabla v\|\|q\|}\geq \beta >0.
\end{equation}
\end{proposition}

\begin{lemma}(A Discrete Gronwall lemma see Lemma 5.1 p.369 
\cite{heywood1990})\label{gronwall}
Let $\Delta t, B, a_n, b_n, c_n, d_n $ be non-negative numbers such that for $l\geq 1$
\begin{equation*}
a_l+\Delta t\sum_{n=0}^lb_n\leq \Delta t \sum_{n=0}^{l-1}d_na_n+\Delta t\sum_{n=0}^lc_n+B, \quad\text{for} \;l\geq 0,
\end{equation*}
then for all $\Delta t>0$,
\begin{equation*}
a_l+\Delta t\sum_{n=0}^lb_n\leq \exp(\Delta t\sum_{n=0}^{l-1}d_n)\Big(\Delta t\sum_{n=0}^lc_n+B\Big),\quad \text{for}\; l\geq 0.
\end{equation*}
\end{lemma}
\section{Stability of Backward Euler}\label{stable-analysis}
This section establishes conditions for stability for the variable $\epsilon$ first-order method with constant time-step:
\begin{equation}\label{discrete-stokes-nonlinear2}
\frac{\bu^{n+1}-\bu^n}{k}+\bu^*\cdot \nabla \bu^{n+1}+\frac{1}{2}(\nabla\cdot\bu^*)\bu^{n+1} - \nu\Dt \bu^{n+1} - \nabla (\frac{1}{\epsilon_{n+1}}\nabla\cdot \bu^{n+1}) = f^{n+1}.
\end{equation}
We prove that the velocity is unconditionally stable, but $\|\bu_t\|$ is stable with restrictions on the change of $\epsilon$.

\subsection{Stability of the velocity}
\begin{theorem}\label{stability of u}
(Stability of variable $\epsilon$ penalty method). The variable $\epsilon$ first-order method \eqref{discrete-stokes-nonlinear} is stable. For any $M>0$, the energy equality holds:
\begin{gather*}
    \frac{1}{2}\int_\Omega |\bu^M|^2 dx+\sum_{n=0}^{M-1}\int_\Omega\left(\frac{1}{2}|\bu^{n+1}-\bu^n|^2+k\nu|\nabla\bu^{n+1}|^2+\frac{k}{\epsilon_{n+1}}|\nabla\cdot\bu^{n+1}|^2\right) dx\\
    =\frac{1}{2}\int_\Omega |\bu^0|^2 dx+\sum_{n=0}^{M-1}k\int_\Omega \bu^{n+1}\cdot f^{n+1} dx,
\end{gather*}
and the stability bound holds:
\begin{gather*}
    \frac{1}{2}\int_\Omega |\bu^M|^2 dx+\sum_{n=0}^{M-1}\int_\Omega\left(\frac{1}{2}|\bu^{n+1}-\bu^n|^2+\frac{k\nu}{2}|\nabla\bu^{n+1}|^2+\frac{k}{\epsilon_{n+1}}|\nabla\cdot\bu^{n+1}|^2\right) dx\\
    \leq \frac{1}{2}\int_\Omega |\bu^0|^2 dx +\sum_{n=0}^{M-1}\frac{k}{2\nu}\|f^{n+1}\|_{-1}^2.
\end{gather*}
\end{theorem}
\begin{proof}
Consider the constant time-step of \eqref{discrete-stokes-nonlinear2} $k_{n+1}=k$ for all n, set $M=T/k$.
Take the $L^2$ inner product of \eqref{discrete-stokes-nonlinear2} with $\bu^{n+1}$. We obtain
\begin{equation*}
 \frac{1}{k}\left(||\bu^{n+1}||^2-(\bu^{n},\bu^{n+1})\right)+\nu||\nabla \bu^{n+1}||^2 +\frac{1}{\epsilon_{n+1}}||\nabla\cdot \bu^{n+1}||^2=(f^{n+1},\bu^{n+1}).
\end{equation*}
Apply the polarization identity \eqref{polar} to the term $(\bu^n,\bu^{n+1})$
\begin{equation*}
 \frac{1}{2k}(||\bu^{n+1}||^2-||\bu^n||^2+||\bu^{n+1}-\bu^n||^2)+\nu||\nabla \bu^{n+1}||^2 +\frac{1}{\epsilon_{n+1}}||\nabla\cdot \bu^{n+1}||^2=(f^{n+1},\bu^{n+1}).
\end{equation*}
By the definition of the dual norm and Young's inequality,
\begin{equation*}
 \frac{1}{2k}(||\bu^{n+1}||^2-||\bu^n||^2+||\bu^{n+1}-\bu^n||^2)+\nu||\nabla \bu^{n+1}||^2 +\frac{1}{\epsilon_{n+1}}||\nabla\cdot \bu^{n+1}||^2\leq \frac{1}{2\nu}||f^{n+1}||^2_{-1}+ \frac{\nu}{2}||\nabla \bu^{n+1}||^2.
\end{equation*}
Sum from $n=0,...,M-1$
\begin{equation*}
\frac{1}{2k}||\bu^{M}||^2+\sum_{n=0}^{M-1}(\frac{1}{2k}||\bu^{n+1}-\bu^n||^2+
\frac{\nu}{2}||\nabla \bu^{n+1}||^2+\frac{1}{\epsilon_{n+1}}||\nabla\cdot \bu^{n+1}||^2)\leq  \frac{1}{2k}||\bu^0||^2+\sum_{n=0}^{M-1} \frac{1}{2\nu}||f^{n+1}||^2_{-1}.
\end{equation*}
Multiply by $2k$ and drop positive terms on the left hand side
\begin{equation*}
||\bu^{M}||^2\leq  ||\bu^0||^2+2k\sum_{n=0}^{M-1}\frac{1}{2\nu}||f^{n+1}||^2_{-1}.
\end{equation*}
\end{proof}
\subsection{Stability of $\|\bu_t\|$ for the linear Stokes problem}
 In order to ensure $\nabla\cdot\bu\to 0$ as $\epsilon\to 0$, we need to bound $\|p_\epsilon\|$ following the idea in Fiordilino \cite{fiordilino2018pressure}. By using the $LBB$ inf-sup condition \eqref{lbb}:
\begin{align*}
    \beta\|p\|&\leq\sup_{v\in X}\frac{(p,\nabla\cdot v)}{\|\nabla v\|} 
    =\sup_{v\in X}\frac{-(f,v)+(\bu_t,v)+(\bu\cdot\nabla\bu,v)+\nu(\nabla\bu,\nabla v)}{\|\nabla v\|} \\
    &\leq \|f\|_{-1}+\|\bu_t\|_{-1}+C\|\nabla\bu\|^2+\nu\|\nabla\bu\|,
\end{align*}
this implies we must begin with a bound of $\|\bu_t\|_{-1}$. 
\begin{remark}
The stability conditions on $\epsilon$ that are derived from the linear Stokes problem are necessary for the case of nonlinear NSE. The stability analysis of the nonlinear term of NSE will be more involved but not alter the fundamental approach of this proof. Hence, this case shall be omitted.
\end{remark}
Consider the first-order method with penalty:
\begin{equation}\label{stokes}
    \frac{\bu^{n+1}-\bu^n}{k}-\nu\Dt\bu^{n+1}-\nabla\left(\frac{1}{\epsilon_{n+1}}\nabla\cdot\bu^{n+1}\right)=f^{n+1}.
\end{equation}
\begin{theorem} (0-stability of linear Stokes)
For any $0\leq n\leq M-1$,  if there is some constant $\alpha$ such that $0\leq \alpha k< 1$ and  $(1-k\alpha)\epsilon_n\leq\epsilon_{n+1}$ holds
, then the following stability bound holds
\begin{align*}
    \sum_{n=0}^{M-1}\left(\frac{k}{2}\|\frac{\bu^{n+1}-\bu^n}{k}\|^2+\frac{\nu}{2}\|\nabla(\bu^{n+1}-\bu^n)\|^2+\frac{1}{2\epsilon_{n+1}}\|\nabla\cdot (\bu^{n+1}-\bu^n)\|^2\right)+\frac{\nu}{2}\|\nabla\bu^M\|^2+\frac{1}{2\epsilon_M}\|\nabla\cdot \bu^M\|^2\\
    \leq\exp(\alpha T)\left\{\frac{\nu}{2}\|\nabla\bu^0\|^2+\frac{1}{2\epsilon_0}\|\nabla\cdot\bu^0\|^2+\sum_{n=0}^{M-1}\frac{k}{2}\|f^{n+1}\|^2\right\}.
\end{align*}
\end{theorem}
\begin{remark}
If $\alpha=0$ in \eqref{epsilon-restriction}, i.e., if $\epsilon_{n+1}\geq\epsilon_n$ for all n, then we have unconditional stability. 
\end{remark}
\begin{proof}
Take the $L^2$ inner product of \eqref{stokes} with $\bu^{n+1}-\bu^n$,
\begin{equation*}
    \frac{1}{k}\|\bu^{n+1}-\bu^n\|^2+\nu\left(\nabla\bu^{n+1},\nabla(\bu^{n+1}-\bu^n)\right)+\frac{1}{\epsilon_{n+1}}\left(\nabla\cdot\bu^{n+1},\nabla\cdot(\bu^{n+1}-\bu^n)\right)=(f^{n+1},\bu^{n+1}-\bu^n).
\end{equation*}
We will address terms successively.
Denote $\gamma_{n+1}=1/\epsilon_{n+1}$ and apply the polarization identity \eqref{polar} to the second and the third terms on the left,
\begin{align*}
    \nu\left(\nabla\bu^{n+1},\nabla(\bu^{n+1}-\bu^n)\right)&=\frac{\nu}{2}\left(\|\nabla\bu^{n+1}\|^2-\|\nabla\bu^n\|^2+\|\nabla(\bu^{n+1}-\bu^n)\|^2\right),\\
    \frac{1}{\epsilon_{n+1}}\left(\nabla\cdot\bu^{n+1},\nabla\cdot(\bu^{n+1}-\bu^n)\right)&=\frac{\gamma_{n+1}}{2}\left(\|\nabla\cdot\bu^{n+1}\|^2-\|\nabla\cdot\bu^n\|^2+\|\nabla\cdot(\bu^{n+1}-\bu^n)\|^2\right).
\end{align*}
By adding and subtracting $\gamma_n\|\nabla\cdot\bu^n\|^2/2$, we have
\begin{align*}
    \frac{\gamma_{n+1}}{2}\left(\|\nabla\cdot\bu^{n+1}\|^2-\|\nabla\cdot\bu^n\|^2+\|\nabla\cdot(\bu^{n+1}-\bu^n)\|^2\right)
    &=\frac{\gamma_{n+1}}{2}\|\nabla\cdot\bu^{n+1}\|^2-\frac{\gamma_n}{2}\|\nabla\cdot\bu^n\|^2\\
    &+\frac{\gamma_{n+1}}{2}\|\nabla\cdot(\bu^{n+1}-\bu^n)\|^2+\frac{\gamma_n-\gamma_{n+1}}{2}\|\nabla\cdot\bu^n\|^2.
\end{align*}
By Cauchy-Schwarz and Young's inequalities \eqref{holder+young}
\begin{equation*}
    (f^{n+1},\bu^{n+1}-\bu^n)\leq\frac{k}{2}\|f^{n+1}\|^2+\frac{1}{2k}\|\bu^{n+1}-\bu^n\|^2.
\end{equation*}
By combining similar terms, we have
\begin{align*}
    \frac{1}{2k}\|\bu^{n+1}-\bu^n\|^2+\left[\left(\frac{\nu}{2}\|\nabla\bu^{n+1}\|^2+\frac{\gamma_{n+1}}{2}\|\nabla\cdot\bu^{n+1}\|^2\right)-\left(\frac{\nu}{2}\|\nabla\bu^n\|^2+\frac{\gamma_n}{2}\|\nabla\cdot\bu^n\|^2\right)\right]\\
    +\frac{\nu}{2}\|\nabla(\bu^{n+1}-\bu^n)\|^2+\frac{\gamma_{n+1}}{2}\|\nabla\cdot(\bu^{n+1}-\bu^n)\|^2+\frac{\gamma_n-\gamma_{n+1}}{2}\|\nabla\cdot\bu^n\|^2
    \leq\frac{k}{2}\|f^{n+1}\|^2.
\end{align*}
{
Moving $(\gamma_n-\gamma_{n+1})/2\|\nabla\cdot \bu^n\|^2$ to the right. We obtain
\begin{align*}
    \frac{1}{2k}\|\bu^{n+1}-\bu^n\|^2+\left[\left(\frac{\nu}{2}\|\nabla\bu^{n+1}\|^2+\frac{\gamma_{n+1}}{2}\|\nabla\cdot\bu^{n+1}\|^2\right)-\left(\frac{\nu}{2}\|\nabla\bu^n\|^2+\frac{\gamma_n}{2}\|\nabla\cdot\bu^n\|^2\right)\right]\\
    +\frac{\nu}{2}\|\nabla(\bu^{n+1}-\bu^n)\|^2+\frac{\gamma_{n+1}}{2}\|\nabla\cdot(\bu^{n+1}-\bu^n)\|^2\leq\frac{k}{2}\|f^{n+1}\|^2+\frac{\gamma_{n+1}-\gamma_n}{2}\|\nabla\cdot\bu^n\|^2,\\
    =\frac{k}{2}\|f^{n+1}\|^2+k\left(\frac{\gamma_{n+1}-\gamma_n}{k\gamma_n}\right)\left(\frac{\gamma_n}{2}\|\nabla\cdot\bu^n\|^2\right).
\end{align*}
For each fixed constant $\alpha\geq 0$, we need $(\gamma_{n+1}-\gamma_n)/k\gamma_n\leq\alpha$ to avoid catastrophic growth. This leads to
\begin{align*}
    \gamma_{n+1}-\gamma_n\leq k\alpha\gamma_n,\quad
    \left(\frac{1}{\epsilon_{n+1}}-\frac{1}{\epsilon_n}\right)\epsilon_n\epsilon_{n+1}\leq k\alpha\frac{1}{\epsilon_n}\epsilon_n\epsilon_{n+1},\quad
    \epsilon_n\leq (1+k\alpha)\epsilon_{n+1},\quad
    \frac{1}{1+k\alpha}\epsilon_n\leq\epsilon_{n+1}.
\end{align*}
If $k\alpha  <1$, we approximate with the first two terms of the Taylor expansion
\begin{equation*}
    \frac{1}{1+k\alpha}  \geq 1-k\alpha
\end{equation*}
Thus we have the stability condition on $\epsilon$
\begin{equation}\label{epsilon-restriction}
    (1-k\alpha)\epsilon_n\leq\epsilon_{n+1}.
\end{equation}
Under condition \eqref{epsilon-restriction}
\begin{equation}\label{stability-ineq}
\begin{aligned}
    \frac{1}{2k}\|\bu^{n+1}-\bu^n\|^2+\left[\left(\frac{\nu}{2}\|\nabla\bu^{n+1}\|^2+\frac{1}{2\epsilon_{n+1}}\|\nabla\cdot\bu^{n+1}\|^2\right)-\left(\frac{\nu}{2}\|\nabla\bu^n\|^2+\frac{1}{2\epsilon_n}\|\nabla\cdot\bu^n\|^2\right)\right]\\
    +\frac{\nu}{2}\|\nabla(\bu^{n+1}-\bu^n)\|^2+\frac{1}{2\epsilon_{n+1}}\|\nabla\cdot(\bu^{n+1}-\bu^n)\|^2\leq\frac{k}{2}\|f^{n+1}\|^2+k\alpha\left(\frac{1}{2\epsilon_n}\|\nabla\cdot\bu^n\|^2\right).
\end{aligned}
\end{equation}
Sum from $n=0,1,\dots, M-1$
\begin{align*}
    \sum_{n=0}^{M-1}\left(\frac{1}{2k}\|\bu^{n+1}-\bu^n\|^2+\frac{\nu}{2}\|\nabla(\bu^{n+1}-\bu^n)\|^2+\frac{1}{2\epsilon_{n+1}}\|\nabla\cdot (\bu^{n+1}-\bu^n)\|^2\right)+\frac{\nu}{2}\|\nabla\bu^M\|^2+\frac{1}{2\epsilon_M}\|\nabla\cdot \bu^M\|^2\\
    \leq\frac{\nu}{2}\|\nabla\bu^0\|^2+\frac{1}{2\epsilon_0}\|\nabla\cdot\bu^0\|^2+\sum_{n=0}^{M-1}\frac{k}{2}\|f^{n+1}\|^2+k\sum_{n=0}^{M-1}\alpha\left(\frac{1}{2\epsilon_n}\|\nabla\cdot\bu^n\|^2\right).
\end{align*}
Apply the Gronwall inequality \eqref{gronwall}
\begin{align*}
    \sum_{n=0}^{M-1}\left(\frac{1}{2k}\|\bu^{n+1}-\bu^n\|^2+\frac{\nu}{2}\|\nabla(\bu^{n+1}-\bu^n)\|^2+\frac{1}{2\epsilon_{n+1}}\|\nabla\cdot (\bu^{n+1}-\bu^n)\|^2\right)+\frac{\nu}{2}\|\nabla\bu^M\|^2+\frac{1}{2\epsilon_M}\|\nabla\cdot \bu^M\|^2\\
    \leq \exp(k\sum_{n=0}^{M-1}\alpha)\left\{\frac{\nu}{2}\|\nabla\bu^0\|^2+\frac{1}{2\epsilon_0}\|\nabla\cdot\bu^0\|^2+\sum_{n=0}^{M-1}\frac{k}{2}\|f^{n+1}\|^2\right\},\\
    =\exp(\alpha T)\left\{\frac{\nu}{2}\|\nabla\bu^0\|^2+\frac{1}{2\epsilon_0}\|\nabla\cdot\bu^0\|^2+\sum_{n=0}^{M-1}\frac{k}{2}\|f^{n+1}\|^2\right\}.
\end{align*}
Thus we proved that, if $(1-k\alpha)\epsilon_n\leq\epsilon_{n+1}$ for some $\alpha\geq0$ and $\alpha k<1$, stability of  discrete $\bu_t$ holds.

}
\end{proof}
\begin{remark}
When $\epsilon$ decreases, \eqref{epsilon-restriction} is needed to ensure the boundedness of discrete $\|\bu_t\|$; If \eqref{epsilon-restriction} does not hold, $\|\bu_t\|$ may have catastrophic growth see \Cref{fig:ut_spike}.
\end{remark}
\begin{theorem}
Let $\bu$ be the solution to penalized NSE \eqref{pnse1}-\eqref{pnse1b}, then $\bu_t\in L^{4/3}(0,T;H^{-1})$, equivalently
\begin{equation*}
    \int_0^T\|\bu_t\|_{-1}^{4/3}\ dt<C( \bu^0, f, k, \nu, T,\min_{t^*\in [0,T]}\epsilon(t^*)).
\end{equation*}
\end{theorem}
\begin{proof}
Recall
\begin{equation*}
    \|\bu_t\|_{-1}=\sup_{v\in X}\frac{(\bu_t,v)}{\|\nabla v\|}.
\end{equation*}
By skew-symmetry 
\begin{align*}
    (\bu_t,v)=-\int_\Omega b^*(\bu,\bu,v) dx-\nu(\nabla\bu,\nabla v)-\frac{1}{\epsilon(t)}(\nabla\cdot\bu,\nabla\cdot v)+(f,v),\\
    \leq C(\|\bu\|^{1/2}\|\nabla\bu\|^{1/2})\|\nabla v\|\|\nabla\bu\|+\nu\|\nabla\bu\|\|\nabla v\|+\frac{1}{\epsilon(t)}\|\nabla\cdot\bu\|\|\nabla v\|+\|f\|_{-1}\|\nabla v\|.
\end{align*}
Thus,
\begin{align*}
    \frac{(\bu_t,v)}{\|\nabla v\|}\leq C\|\bu\|^{1/2}\|\nabla\bu\|^{3/2}+\nu\|\nabla\bu\|+\frac{1}{\epsilon(t)}\|\nabla\cdot\bu\|+\|f\|_{-1}.
\end{align*}
$\|\bu\|$ is bounded by the problem data and initial condition from the stability of the velocity Theorem 3.1.
\begin{equation*}
      \|\bu_t\|_{-1}\leq C(\bu^0, f, k, \nu)\|\nabla\bu\|^{3/2}+\nu\|\nabla\bu\|+\frac{1}{\epsilon(t)}\|\nabla\cdot\bu\|+\|f\|_{-1}.
\end{equation*}
Then
\begin{align*}
    \int_0^T \|\bu_t\|_{-1}^{4/3}\ dt\leq C( \bu^0, f, k, \nu)\int_0^T \|\nabla\bu\|^2\ dt+C(\nu)\int_0^T\|\nabla\bu\|^{4/3}\ dt\\
    +C\int_0^T\left(\frac{1}{\epsilon(t)}\|\nabla\cdot \bu\|\right)^{4/3}dt+C\int_0^T\|f\|_{-1}^{4/3}dt.
\end{align*}
From Theorem 3.1 the stability bound
\begin{equation*}
    \int_0^T\|\nabla\bu\|^2\ dt<C( \bu^0, f, k, \nu),\quad \text{and}\quad \int_0^T\frac{1}{\epsilon(t)}\|\nabla\cdot \bu\|^2\ dt<C( \bu^0, f, k, \nu).
\end{equation*}
By H\"{o}lder's inequality \eqref{holder+young}
\begin{align*}
    \int_0^T\|\nabla\bu\|^{4/3}\ dt&\leq \left(\int_0^T 1^3\ dt\right)^{1/3}\left(\int_0^T (\|\nabla\bu\|^{4/3})^{3/2}\ dt\right)^{2/3}=C(T)\left(\int_0^T\|\nabla\bu\|^2\ dt\right)^{2/3},\\
    \int_0^T\left(\frac{1}{\epsilon(t)}\|\nabla\cdot \bu\|\right)^{4/3}\ dt&\leq \max_{t^*\in[0,T]}\left(\frac{1}{\epsilon(t^*)}\right)^{2/3}\left(\int_0^T 1^3\ dt\right)^{1/3} \left(\int_0^T\left(\frac{1}{\epsilon(t)^{2/3}}\|\nabla\cdot\bu\|^{4/3}\right)^{3/2}\ dt\right)^{2/3},\\
    &=C(T,\min_{t^*\in [0,T]} \epsilon(t^*))\left(\int_0^T \frac{1}{\epsilon(t)}\|\nabla\cdot\bu\|^2\ dt\right)^{2/3}.
\end{align*}
Then the result follows.
\end{proof}

\section{Error Analysis} Next, we will prove an error estimate for the semi-discrete, variable-$\epsilon$ penalty method.
Find $(\bu_\epsilon^h,p_\epsilon^h)\in (X^h,Q^h)$ such that 
\begin{equation}\label{eq:vform}
\begin{aligned}
    (\bu_{\epsilon,t}^h,v^h)+b^*(\bu_\epsilon^h,\bu_\epsilon^h,v^h)+\nu(\nabla \bu_\epsilon^h,\nabla v^h)-(p_\epsilon^h,\nabla\cdot v^h)+(q^h,\nabla\cdot\bu_\epsilon^h)+\epsilon(t)(p_\epsilon^h,q^h)=(f,v^h),
\end{aligned}
\end{equation}
for all $(v^h,q^h)\in (X^h,Q^h)$.
\begin{definition}
(Stokes Projection \cite{stokes-projection-labovsky2009stabilized}) The Stokes projection operator \\
$P_S:(X,Q)\to (X^h,Q^h),P_S(\bu,p)=(\Tilde{\bu},\Tilde{p})$, satisfies
\begin{equation}\label{stokes-proj}
\begin{aligned}
    \nu(\nabla(\bu-\Tilde{\bu}),\nabla v^h)-(p-\Tilde{p},\nabla\cdot v^h)&=0,\\
        (\nabla\cdot(\bu-\Tilde{\bu}), q^h)&=0,
\end{aligned}
\end{equation}
for any $v^h\in X^h, q^h\in Q^h$.
\end{definition}
\begin{proposition}
(Error estimate for the Stokes Projection) Suppose the discrete inf-sup condition \eqref{lbbh} holds. Let $C_1$ be a constant independent of $h$ and $\nu$ and $C_2=C(\nu,\Omega)$. If $\Omega$ is a convex polygonal/polyhedral domain, then the error in the Stokes Projection \eqref{stokes-proj} satisfies
\begin{gather*}
    \|p-\tilde{p}\|\leq \frac{\nu}{\beta^h}\|\nabla(\bu-\Tilde{\bu})\|,\\
    \nu\|\nabla(\bu-\Tilde{\bu})\|^2\leq C_1[\nu\inf_{v^h\in X^h}\|\nabla(\bu-v^h)\|^2+\nu^{-1}\inf_{q^h\in Q^h}\|p-q^h\|^2],\\
    \text{and}\ \|\bu-\Tilde{\bu}\|\leq C_2h\left(\inf_{v^h\in X^h}\|\nabla(\bu-v^h)\|+\inf_{q^h\in Q^h}\|p-q^h\|\right).
\end{gather*}
\end{proposition}
\begin{proof}
From the first equation of \eqref{stokes-proj},
\begin{equation*}
    (p-\Tilde{p},\nabla\cdot v^h)=\nu(\nabla(\bu-\Tilde{\bu}),\nabla v^h).
\end{equation*}
By the discrete inf-sup condition \eqref{lbbh},
\begin{align*}
    \beta^h\|p-\tilde{p}\|&\leq \sup_{v^h\in X^h}\frac{(p-\Tilde{p},\nabla\cdot v^h)}{\|\nabla v^h\|}=\sup_{v^h\in X^h}\frac{\nu(\nabla(\bu-\Tilde{\bu}),\nabla v^h)}{\|\nabla v^h\|},\\
    &\leq\sup_{v^h\in X^h}\frac{\nu\|\nabla(\bu-\Tilde{\bu})\|\|\nabla v^h\|}{\|\nabla v^h\|}=\nu \|\nabla (\bu-\Tilde{\bu})\|.
\end{align*}
For detailed proof of the last two inequalities, see Proposition 2.2 and Remark 2.2 of \cite{stokes-projection-labovsky2009stabilized}.
\end{proof}

We also need an estimator for $\|\nabla(\bu-\Tilde{\bu})_t\|$. Take the partial derivative with respect to time $t$ of \eqref{stokes-proj} to yield
\begin{align*}
    \nu(\nabla(\bu-\Tilde{\bu})_t,\nabla v^h)-((p-\Tilde{p})_t,\nabla\cdot v^h)=0,\\
    (\nabla\cdot(\bu-\Tilde{\bu})_t,q^h)=0,
\end{align*}
for all $v^h\in X^h, q^h\in Q^h$.

Let $v^h=\phi^h_t, q^h=(\Tilde{p}-I(p))_t$, by a similar argument as in Proposition 4.2, we have
\begin{equation}\label{eta-t}
\begin{aligned}
     \nu\|\nabla(\bu-\Tilde{\bu})_t\|^2
     \leq C[\nu\inf_{v^h\in X^h}\|\nabla(\bu-v^h)_t\|^2
     +\nu^{-1}
     \inf_{q^h\in Q^h}\|(p-q^h)_t\|^2],\\
\end{aligned}
\end{equation}
where $C$ is a constant independent of $h$ and $\nu$.
\begin{theorem}\label{error analysis}
(Error Analysis of semi-discrete variable $\epsilon$ penalty method) Let $(X^h, Q^h)$ be the finite element spaces satisfying \eqref{approx-prop} and \eqref{lbbh}. Let $\bu_\epsilon$ be a solution of \eqref{pnse1}. Suppose the interpolation estimate \eqref{negative-norm} in $H^{-1}(\Omega)$ holds and $\|\nabla\bu_\epsilon\|\in L^4(0,T)$, then we have the following error estimate:
\begin{gather*}
    \sup_{0\leq t\leq T}\|(\bu_\epsilon-\bu_\epsilon^h)(t)\|^2+\int_0^T \frac{\nu}{4}\|\nabla(\bu_\epsilon-\bu_\epsilon^h)\|^2\ dt \leq\\
     e^{\int_0^T a(t) dt} \Big\{\|(\bu_\epsilon-\bu_\epsilon^h)(0)\|^2+\int_0^T C(\nu,\beta^h)\epsilon(t)h^{2m}\left(\|\bu_\epsilon\|^2_{H^{m+1}(\Omega)}+\|p_\epsilon\|^2_{H^m(\Omega)}\right)\ dt
    +\max_{0\leq t\leq T}\|\bu_\epsilon-v^h\|^2\Big\}\\
    +C(\nu,\Omega)\Big[(h^{-d/3}+h^{2-d/3}+h^{2-d/2})h^{2m}\left(\|\bu_\epsilon\|^2_{L^2(0,T;H^{m+1}(\Omega))}+\|p_\epsilon\|^2_{L^2(0,T;H^m(\Omega))}\right)\\
    +h^{2m+2}\left(\|\bu_{\epsilon,t}\|^2_{L^2(0,T;H^{m+1}(\Omega))}+\|p_{\epsilon,t}\|^2_{L^2(0,T;H^m(\Omega))}\right)\Big]\Big\},\\
    \text{where}\ a(t)=C(\nu)\|\nabla\bu_\epsilon\|^4+\frac{1}{4} .
\end{gather*}
\end{theorem} 
\begin{proof}
We denote $(\bu_\epsilon, p_\epsilon)$ as penalty solutions to \eqref{pnse1}.\\ 
Multiplying first equation of \eqref{pnse1} by $v^h\in X^h$ and second equation of \eqref{pnse1} by $q^h\in Q^h$ gives
\begin{align}\label{eq:cont}
    (\bu_{\epsilon,t},v^h)+b^*(\bu_\epsilon,\bu_\epsilon,v^h)+\nu(\nabla \bu_\epsilon,\nabla v^h)-(p_\epsilon,\nabla\cdot v^h)+(q^h,\nabla\cdot\bu_\epsilon)+\epsilon(t)(p_\epsilon,q^h)=(f,v^h).
\end{align}
Subtract \eqref{eq:vform} from \eqref{eq:cont} and denote $e=\bu_\epsilon-\bu_\epsilon^h$,
\begin{align*}
    (e_t,v^h)+b^*(\bu_\epsilon,\bu_\epsilon,v^h)-b^*(\bu_\epsilon^h,\bu_\epsilon^h,v^h)+\nu(\nabla e,\nabla v^h)\\
    -(p_\epsilon-p_\epsilon^h,\nabla\cdot v^h)+(q^h,\nabla\cdot e)+\epsilon(t)(p_\epsilon-p_\epsilon^h,q^h)=0.
\end{align*}
Denote $\eta=\bu_\epsilon-\tilde{\bu}, \phi^h=\bu_\epsilon^h-\tilde{\bu},e=\eta-\phi^h$ and $ \tilde{\bu}\in X^h, \lambda^h\in Q^h$,
\begin{gather*}
    (\phi_t^h,v^h)+\nu(\nabla\phi^h,\nabla v^h)-(p_\epsilon^h-\lambda^h,\nabla\cdot v^h)+(q^h,\nabla\cdot\phi^h)+\epsilon(t)(p_\epsilon^h-\lambda^h,q^h)\\
    =(\eta_t,v^h)+\nu(\nabla\eta,\nabla v^h)-(p_\epsilon-\lambda^h,\nabla\cdot v^h)+(q^h,\nabla\cdot\eta)+\epsilon(t)(p_\epsilon-\lambda^h,q^h)\\
    +b^*(\bu_\epsilon,\bu_\epsilon,v^h)-b^*(\bu_\epsilon^h,\bu_\epsilon^h,v^h).
\end{gather*}
Pick $\tilde{\bu} \in X^h, \lambda^h \in Q^h$ to be the Stokes Projection \eqref{stokes-proj} of $(\bu_\epsilon,p_\epsilon)$ such that 
\begin{align*}
     \nu(\nabla(\bu_\epsilon-\tilde{\bu}),\nabla v^h)-(p_\epsilon-\lambda^h,\nabla\cdot v^h)=0& \text{ for all } v^h\in X^h ,\\
     (\nabla\cdot(\bu_\epsilon-\Tilde{\bu}),q^h)=0& \text{ for all } q^h\in Q^h .
\end{align*}
Set $v^h=\phi^h, q^h=p_\epsilon^h-\lambda^h$. We obtain,
\begin{equation*}
    \frac{1}{2}\frac{d}{dt}\|\phi^h\|^2+\nu\|\nabla\phi^h\|^2+\epsilon(t)\|p_\epsilon^h-\lambda^h\|^2=(\eta_t,\phi^h)+b^*(\bu_\epsilon,\bu_\epsilon,\phi^h)-b^*(\bu_\epsilon^h,\bu_\epsilon^h,\phi^h)+\epsilon(t)(p_\epsilon-\lambda^h,p_\epsilon^h-\lambda^h).
\end{equation*}
Consider the nonlinear terms
\begin{align*}
    b^*(\bu_\epsilon,\bu_\epsilon,\phi^h)-b^*(\bu_\epsilon^h,\bu_\epsilon^h,\phi^h)&=b^*(\bu_\epsilon,\bu_\epsilon,\phi^h)-b^*(\bu_\epsilon^h,\bu_\epsilon,\phi^h)+b^*(\bu_\epsilon^h,\bu_\epsilon,\phi^h)-b^*(\bu_\epsilon^h,\bu_\epsilon^h,\phi^h)\\
    =b^*(e,\bu_\epsilon,\phi^h)+b^*(\bu_\epsilon^h,e,\phi^h)
    &=b^*(\eta,\bu_\epsilon,\phi^h)-b^*(\phi^h,\bu_\epsilon,\phi^h)+b^*(\bu_\epsilon^h,\eta,\phi^h).
\end{align*}
Thus we have
\begin{align*}
    \frac{1}{2}\frac{d}{dt}\|\phi^h\|^2+\nu\|\nabla\phi^h\|^2+\epsilon(t)\|p_\epsilon^h-\lambda^h\|^2
    &=(\eta_t,\phi^h)+b^*(\eta,\bu_\epsilon,\phi^h)-b^*(\phi^h,\bu_\epsilon,\phi^h)+b^*(\bu_\epsilon^h,\eta,\phi^h)\\
    &+\epsilon(t)(p_\epsilon-\lambda^h,p_\epsilon^h-\lambda^h).
\end{align*}
Consider the right hand side terms of the equation 
\begin{align*}
    |(\eta_t,\phi^h)|\leq \frac{1}{2\nu}\|\eta_t\|_{-1}^2+\frac{\nu}{2}\|\nabla\phi^h\|^2,\quad
    |\epsilon(t)(p_\epsilon-\lambda^h,p_\epsilon^h-\lambda^h)|\leq\frac{\epsilon(t)}{2}\|p_\epsilon-\lambda^h\|^2+\frac{\epsilon(t)}{2}\|p_\epsilon^h-\lambda^h\|.
\end{align*}
Apply the trilinear inequality \eqref{trilinear ineq} to the first nonlinear term $b^*(\eta,\bu_\epsilon,\phi^h)$ to obtain
\begin{align*}
    |b^*(\eta,\bu_\epsilon,\phi^h)|&=\frac{1}{2}\left|(\eta\cdot\nabla\bu_\epsilon,\phi^h)-(\eta\cdot\nabla\phi^h,\bu_\epsilon)\right|
    \leq\frac{1}{2}[\|\eta\|_{L_4}\|\nabla\bu_\epsilon\|\|\phi^h\|_{L_4}+\|\eta\|_{L_4}\|\nabla\phi^h\|\|\bu_\epsilon\|_{L_4}].
\end{align*}
Using the Sobolev inequality \eqref{sobolev}, we have $\|\phi^h\|_{L_4}\leq C\|\nabla\phi^h\|$ and $\|\bu_\epsilon\|_{L_4}\leq C\|\nabla\bu_\epsilon\|$,
\begin{align*}
    |b^*(\eta,\bu_\epsilon,\phi^h)|
    &\leq C\|\eta\|_{L_4}\|\nabla\bu_\epsilon\|\|\nabla\phi^h\| 
    \leq \frac{\nu}{4}\|\nabla\phi^h\|^2+C(\nu)\|\nabla\bu_\epsilon\|^2\|\eta\|_{L_4}^2.
\end{align*}
Apply Lemma \eqref{skewsymsharper} to the term $|b^*(\phi^h,\bu_\epsilon,\phi^h)|$,
\begin{equation*}
    |b^*(\phi^h,\bu_\epsilon,\phi^h)|\leq C(\Omega)\|\phi^h\|^{1/2}\|\nabla\phi^h\|^{3/2}\|\nabla\bu_\epsilon\|.
\end{equation*}
Using H\"{o}lder's and Young's inequality \eqref{holder+young} with $p=4/3, q=4$,
\begin{equation*}
    |b^*(\phi^h,\bu_\epsilon,\phi^h)|\leq \frac{\nu}{16}\|\nabla\phi\|^2+C(\nu)\|\phi^h\|^2\|\nabla\bu_\epsilon\|^4.
\end{equation*}
Next, we bound the nonlinear term $b^*(\bu_\epsilon^h,\eta,\phi^h)$ and use the trilinear inequality \eqref{trilinear ineq}
\begin{align*}
    |b^*(\bu_\epsilon^h,\eta,\phi^h)|&=\frac{1}{2}\left|(\bu_\epsilon^h\cdot\nabla\eta,\phi^h)-(\bu_\epsilon^h\cdot\nabla\phi^h,\eta)\right| \\
    &\leq \frac{1}{2}[\|\bu_\epsilon^h\|_{L_6}\|\nabla\eta\|_{L_3}\|\phi^h\|+\|\bu_\epsilon^h\|_{L_6}\|\nabla\phi^h\|\|\eta\|_{L_3}] \\
    &\leq \frac{1}{4}\|\phi^h\|^2+\frac{1}{4}\|\bu_\epsilon^h\|_{L_6}^2\|\nabla\eta\|_{L_3}^2+\frac{\nu}{16}\|\nabla\phi^h\|^2+C(\nu)\|\bu_\epsilon^h\|_{L_6}^2\|\eta\|_{L_3}^2.
\end{align*}
Collect all the terms, combine similar terms and multiply through by 2, we have
\begin{align*}
    \frac{d}{dt}\|\phi^h\|^2+\frac{\nu}{4}\|\nabla\phi^h\|^2+\epsilon(t)\|p_\epsilon^h-\lambda^h\|^2\leq (C(\nu)\|\nabla\bu_\epsilon\|^4+\frac{1}{2})\|\phi^h\|^2\\
    +C(\nu)\left[\|\eta_t\|_{-1}^2+\|\nabla\bu_\epsilon\|^2\|\eta\|_{L_4}^2+\|\bu_\epsilon^h\|_{L_6}^2\|\eta\|_{L_3}^2\right]+\frac{1}{2}\|\bu_\epsilon^h\|_{L_6}^2\|\nabla\eta\|_{L_3}^2+\epsilon(t)\|p_\epsilon-\lambda^h\|^2.
\end{align*}
Denote $a(t)=C(\nu)\|\nabla\bu_\epsilon\|^4+\frac{1}{2}$ and its antiderivative
\begin{equation*}
A(T):=\int_0^T a(t) dt<\infty \ \text{for}\  \|\nabla\bu_\epsilon\|\in L^4(0,T).
\end{equation*}
Multiply through by the integrating factor $e^{-A(t)}$ 
\begin{gather*}
    \frac{d}{dt}[e^{-A(T)}\|\phi^h\|^2]+e^{-A(T)}\Big[\frac{\nu}{4}\|\nabla\phi^h\|^2+\epsilon(t)\|p^h-\lambda^h\|^2\Big] \leq \\
    e^{-A(T)}\Big\{C(\nu)\Big[\|\eta_t\|_{-1}^2+\|\nabla\bu_\epsilon\|^2\|\eta\|_{L_4}^2+\|\bu_\epsilon^h\|_{L_6}^2\|\eta\|_{L_3}^2\Big]+\frac{1}{2}\|\bu_\epsilon^h\|_{L_6}^2\|\nabla\eta\|_{L_3}^2+\epsilon(t)\|p_\epsilon-\lambda^h\|^2\Big\}.
\end{gather*}
Integrate over $[0,T]$ and multiply through by $e^{A(T)}$ gives
\begin{gather*}
    \|\phi^h(T)\|^2+\int_0^T\frac{\nu}{4}\|\nabla\phi^h\|^2+\epsilon(t)\|p^h-\lambda^h\|^2\ dt \leq e^{A(T)}\Big\{\|\phi^h(0)\|^2 \\ +\int_0^TC(\nu)\Big[\|\eta_t\|_{-1}^2+\|\nabla\bu_\epsilon\|^2\|\eta\|_{L_4}^2+\|\bu_\epsilon^h\|_{L_6}^2\|\eta\|_{L_3}^2 \Big]+\frac{1}{2}\|\bu_\epsilon^h\|_{L_6}^2\|\nabla\eta\|_{L_3}^2+\epsilon(t)\|p_\epsilon-\lambda^h\|^2\ dt\Big\}.
\end{gather*}
Applying H\"{o}lder's inequality \eqref{holder+young} gives
\begin{align*}
\int_0^T \|\bu_\epsilon^h\|_{L_6}^2\|\eta\|_{L_3}^2 dt &\leq \|\bu_\epsilon^h\|_{L^6(0,T;L^6)}^2\|\eta\|_{L^3(0,T;L^3)}^2, \\
\int_0^T \|\nabla\bu_\epsilon\|^2\|\eta\|_{L_4}^2 dt&\leq \|\nabla\bu_\epsilon\|_{L^4(0,T;L^2)}^2\|\eta\|_{L^4(0,T;L^4)}^2, \\
\int_0^T \|\bu_\epsilon^h\|_{L_6}^2\|\nabla\eta\|_{L_3}^2 dt &\leq \|\bu_\epsilon^h\|_{L^6(0,T;L^6)}^2\|\nabla\eta\|_{L^3(0,T;L^3)}^2 .
\end{align*}
$\|\bu_\epsilon^h\|_{L^6(0,T;L^6)}$ and $\|\nabla\bu_\epsilon\|_{L^4(0,T;L^2)}$ are bounded by problem data by the stability bound. Using the Sobolev inequality \eqref{sobolev2}, $L^p-L^2$ type inverse inequality \eqref{lp-l2inverse}, the interpolation  estimate \eqref{negative-norm} and the Poincar\'{e}-Friedrichs' inequality \eqref{pf ineq}
\begin{align*}
    &\|\eta\|_{L^3}\leq C\|\eta\|^{1-d/6}\|\nabla \eta\|^{d/6},\quad
    \|\eta\|_{L^4}\leq C\|\eta\|^{1-d/4}\|\nabla\eta\|^{d/4},\\
    &\|\nabla\eta\|_{L^3}\leq Ch^{-d/6}\|\nabla \eta\|,\quad\quad\quad
    \|\eta_t\|_{-1}\leq Ch\|\eta_t\|\leq Ch\|\nabla\eta_t\|.
\end{align*}
By Proposition 4.2 and \eqref{eta-t}
\begin{align*}
    \|\nabla(\bu_\epsilon-\Tilde{\bu})\|^2\leq C(\nu)[\inf_{v^h\in X^h}\|\nabla(\bu_\epsilon-v^h)\|^2+\inf_{q^h\in Q^h}\|p_\epsilon-q^h\|^2],\\
    \|\bu_\epsilon-\Tilde{\bu}\|^2\leq C(\nu,\Omega)h^2[\inf_{v^h\in X^h}\|\nabla(\bu_\epsilon-v^h)\|^2+\inf_{q^h\in Q^h}\|p_\epsilon-q^h\|^2],\\
    \|\nabla(\bu_\epsilon-\Tilde{\bu})_t\|^2\leq C(\nu)[\inf_{v^h\in X^h}\|\nabla(\bu_\epsilon-v^h)_t\|^2+\inf_{q^h\in Q^h}\|(p_\epsilon-q^h)_t\|^2],\\
    \|p_\epsilon-\lambda^h\|^2\leq C(\nu,\beta^h)[\inf_{v^h\in X^h}\|\nabla(\bu_\epsilon-v^h)\|^2+\inf_{q^h\in Q^h}\|p_\epsilon-q^h\|^2].
\end{align*}
Thus,
\begin{gather*}
    \|\phi^h(T)\|^2+\int_0^T\frac{\nu}{4}\|\nabla\phi^h\|^2+\epsilon(t)\|p^h-\lambda^h\|^2\ dt \leq \\
     e^{A(T)}\Big\{\|\phi^h(0)\|^2+\int_0^T C(\nu,\beta^h)\epsilon(t)\left(\inf_{v^h\in X^h}\|\nabla(\bu_\epsilon-v^h)\|^2+\inf_{q^h\in Q^h}\|p_\epsilon-q^h\|^2\right)\ dt \\
    +C(\nu,\Omega)\Big[(h^{-d/3}+h^{2-d/3}+h^{2-d/2})\left(\inf_{v^h\in X^h}\|\nabla(\bu_\epsilon-v^h)\|_{L^2(0,T;L^2)}^2+\inf_{q^h\in Q^h}\|p_\epsilon-q^h\|^2_{L^2(0,T;L^2)}\right)\\
    +h^2\left(\inf_{v^h\in X^h}\|\nabla(\bu_\epsilon-v^h)_t\|^2_{L^2(0,T;L^2)}+\inf_{q^h\in Q^h}\|(p_\epsilon-q^h)_t\|^2_{L^2(0,T;L^2)}\right)\Big]
    \Big\}.
\end{gather*}
Using the approximation properties \eqref{approx-prop} of the spaces $(X^h, Q^h)$
\begin{gather*}
     \|\phi^h(T)\|^2+\int_0^T\frac{\nu}{4}\|\nabla\phi^h\|^2+\epsilon(t)\|p^h-\lambda^h\|^2\ dt \leq \\
     e^{A(T)}\Big\{\|\phi^h(0)\|^2++\int_0^T C(\nu,\beta^h)\epsilon(t)h^{2m}\left(\|\bu_\epsilon\|^2_{H^{m+1}(\Omega)}+\|p_\epsilon\|^2_{H^m(\Omega)}\right)\ dt \\
    +C(\nu,\Omega)\Big[(h^{-d/3}+h^{2-d/3}+h^{2-d/2})h^{2m}\left(\|\bu_\epsilon\|^2_{L^2(0,T;H^{m+1}(\Omega))}+\|p_\epsilon\|^2_{L^2(0,T;H^m(\Omega))}\right)\\
    +h^{2m+2}\left(\|\bu_{\epsilon,t}\|^2_{L^2(0,T;H^{m+1}(\Omega))}+\|p_{\epsilon,t}\|^2_{L^2(0,T;H^m(\Omega))}\right)\Big]
    \Big\}.
\end{gather*}
Drop the pressure term on the left-hand side and apply triangle inequality, then we have the error estimate.
\end{proof}
\section{Algorithms}
The backward Euler method was chosen as our method of time discretization. A time filter to increase the accuracy from first to second order was added \cite{guzel2018timefilter}, and later used to implement time adaptivity easily. 

We use the time discretization in \cite{guzel2018timefilter}: 
for $y'=f(t,y)$, select $\tau=k_{n+1}/k_n$, $\alpha=\tau(1+\tau)/(1+2\tau)$, then
\begin{equation}\label{filter}
\begin{aligned}
    &\frac{y_{n+1}^1-y_n}{k_{n+1}}=f(t_{n+1},y_{n+1}),\\
    &y_{n+1}=y_{n+1}^1-\frac{\alpha}{2}\left(\frac{2k_{n}}{k_n+k_{n+1}}y_{n+1}-2y_n+\frac{2k_{n+1}}{k_n+k_{n+1}}y_{n-1}\right) , \\
    &EST=|y_{n+1}-y^1_{n+1}|.
\end{aligned}
\end{equation}
This step uses the information of the previous two time-steps. 

This above algorithm is second-order accurate for $\alpha=2/3$ with constant time-step $\tau=1$. 
Apply this time filter to our adaptive penalty method; we get the following variable $\epsilon$, constant time-step Algorithm 1.
\begin{algorithm}[h!]
\SetAlgoLined
Given $\bu_n,\bu_{n-1},\epsilon_{n},\epsilon_{n+1}$, tolerance TOL, lower tolerance minTOL,  $\epsilon_{min}, \epsilon_{max}$, and $\alpha.$  \\
Set $\bu^\star=2\bu_n-\bu_{n-1}$
Solve for $\bu_{n+1}^1$ 
\begin{align*}
    \frac{\bu_{n+1}^1-\bu_n}{k}+\bu^*\cdot\nabla\bu_{n+1}^1+\frac{1}{2}(\nabla\cdot\bu^*)\bu_{n+1}^1-\nabla\left(\frac{1}{\epsilon_{n+1}}\nabla\cdot\bu_{n+1}^1\right)-\nu\Delta\bu_{n+1}^1=f_{n+1}.
\end{align*}
Apply time filter, Compute estimator EST 
\begin{align*}
    \bu_{n+1}=\bu_{n+1}^1-\frac{1}{3}\{\bu_{n+1}^1-2\bu_n+\bu_{n-1}\}, \\
    EST_{n+1}=\|\nabla\cdot \bu_{n+1}\|/\|\nabla\bu_{n+1}\|. \\
\end{align*}
Adapt $\epsilon$ using the standard decision tree:\\
\If {$EST_{n+1} \geq TOL$} {
    \eIf{$\epsilon_{n+1}=\epsilon_{min}$}
    {CONTINUE}
    {$\epsilon_{n+1} \gets \max\{(1-\alpha k)\epsilon_{n+1},  0.5\epsilon_{n+1},\epsilon_{min}\}$ \;
    REPEAT step}}
\If{$EST_{n+1}\leq minTOL$}{
        $\epsilon_{n+2} \gets \min\{2\epsilon_{n+1},\epsilon_{max}\}$ \;
        CONTINUE \;
    }
Recover pressure $p_{n+1}$ if needed by: 
$
     p_{n+1}=-\frac{1}{\epsilon_{n+1}}\nabla\cdot\bu_{n+1}.
$
\caption{Variable $\epsilon$, constant time-step, second-order penalty method}
\end{algorithm}


 Next, we extend the algorithm to variable time-step methods based on the previous work by Guzel and Layton \cite{guzel2018timefilter} and Layton and McLaughlin \cite{layton2019doublyadaptive}. We summarize as follows.
In the variable time-step, the first-order and second-order method, the next time-step is adapted based on the following: 
\begin{align*}
    \text{first-order prediction}\ k_{new}=k_{old}\left(\frac{tTOL}{tEST_1}\right)^{1/2},\\
    \text{second-order prediction}\  k_{new}=k_{old}\left(\frac{tTOL}{tEST_2}\right)^{1/3}.
\end{align*}
Let $D_2$ denote the difference
\begin{equation*}
    D_2(n+1)=\frac{2k_n}{k_n+k_{n+1}}\bu_{n+1}^1-2\bu_n+\frac{2k_{n+1}}{k_n+k_{n+1}}\bu_{n-1}.
\end{equation*}
A simple estimate of the local truncation error in the first-order estimation is taken to be the difference between $\bu_{n+1}$ and $\bu_{n+1}^1$
\begin{align*}
    \alpha_1=\frac{\tau(1+\tau)}{1+2\tau},\\
    tEST_1=\|\bu_{n+1}-\bu_{n+1}^1\|=\frac{\alpha_1}{2}\|D_2(n+1)\|.
\end{align*}
And the local truncation error of the second-order method is given by
\begin{align*}
    \alpha_2=\frac{\tau_n(\tau_{n+1}\tau_n+\tau_n+1)(4\tau_{n+1}^3+5\tau_{n+1}^2+\tau_{n+1})}{3(\tau_n\tau_{n+1}^2+4\tau_n\tau_{n+1}+2\tau_{n+1}+\tau_n+1)},\\
    tEST_2=\frac{\alpha_2}{6}\left\|\frac{3k_{n-1}}{k_{n+1}+k_n+k_{n-1}}D_2(n+1)-\frac{3k_{n+1}}{k_{n+1}+k_n+k_{n-1}}D_2(n)\right\|.
\end{align*}
For both first-order and second-order variable time-step methods, $\epsilon$ is still adapted independently  using the same decision tree as in Algorithm 1.
\begin{remark}
 The estimator $\|\nabla\cdot\bu_{n+1}\|/\|\nabla\bu_{n+1}\|$ is chosen over $\|\nabla\cdot\bu_{n+1}\|$ as it is dimension free and removes dependence on the size of u.
\end{remark}

Next, we consider the variable time-step variable order method. This algorithm computes two velocity approximations. $\bu^1$ is first-order, and $\bu$ is second-order by applying the time filter. 
The first-order variable time-step method, is unconditionally stable, while the second-order variable time-step method is A-stable, which would require a time-step condition for stability. Combining both first and second order methods by adapting the method order increases accuracy and efficiency.

This following Algorithm 2 gives the variable $\epsilon$, variable time-step variable order (VSVO) penalty method. First (n=1) and second (n=2) order variable time-step method can be also obtained from this following algorithm by using corresponding time-step estimator $tEST_n$ and time-step $STEPn$. In order to use first-order method, $\bu_{n+1}^1$ is used and for second order method, $\bu_{n+1}$ is used instead. For detailed variable $\epsilon$, variable time-step, first and second order algorithms, see the Appendix.

\begin{algorithm}[h!]
\SetAlgoLined
Given $\bu_n,\bu_{n-1},\epsilon_{n+1},\epsilon_n$, tolerance for $\epsilon$: TOL=$10^{-6}$ and lower tolerance minTOL=TOL/10, lower and upper bound of $\epsilon: \epsilon_{min}=10^{-8},\epsilon_{max}=10^{-5}, \alpha=2$, tolerance for $\Delta t$: tTOL=$10^{-5}$ and lower tolerance mintTOL=tTOL/10  \\
\textit{Compute} $\tau=\frac{k_{n+1}}{k_n}$ and $\alpha_1=\frac{\tau(1+\tau)}{1+2\tau},\alpha_2=\frac{\tau_n(\tau_{n+1}\tau_n+\tau_n+1)(4\tau_{n+1}^3+5\tau_{n+1}^2+\tau_{n+1})}{3(\tau_n\tau_{n+1}^2+4\tau_n\tau_{n+1}+2\tau_{n+1}+\tau_n+1)}$ \\
Set $\bu^*=(1+\tau)\bu_n-\tau\bu_{n-1}$ \\
Solve for $\bu_{n+1}^1$ \\
\begin{align*}
    \frac{\bu_{n+1}^1-\bu_n}{k_{n+1}}+\bu^*\cdot\nabla\bu_{n+1}^1+\frac{1}{2}(\nabla\cdot\bu^*)\bu_{n+1}^1-\nabla\left(\frac{1}{\epsilon_{n+1}}\nabla\cdot\bu_{n+1}^1\right)-\nu\Delta\bu_{n+1}^1=f_{n+1}.
\end{align*}
Compute estimators for $\Dt t$ and $\epsilon$ and difference $D_2$ and apply time filter
\begin{align*}
    D_2(n+1)&=\frac{2k_n}{k_n+k_{n+1}}\bu_{n+1}^1-2\bu_n+\frac{2k_{n+1}}{k_n+k_{n+1}}\bu_{n-1}, \\
    \bu_{n+1}&=\bu_{n+1}^1-\frac{\alpha_1}{2}D_2(n+1), \\
    EST_e(n+1)&=\|\nabla\cdot \bu_{n+1}\|/\|\nabla\bu_{n+1}\|, \\
    tEST_1(n+1)&=\frac{\alpha_1}{2}\|D_2(n+1)\|, \\
    tEST_2(n+1)&=\frac{\alpha_2}{6}\left\|\frac{3k_{n-1}}{k_{n+1}+k_n+k_{n-1}}D_2(n+1)-\frac{3k_{n+1}}{k_{n+1}+k_n+k_{n-1}}D_2(n)\right\|.
\end{align*}
Adapt $\epsilon$ and $k$ using the standard decision tree:\\
\eIf {$EST_e(n+1) > TOL$ or $\min\{tEST_1(n+1),tEST_2(n+1)\} > tTOL$} {
    $
    \epsilon_{n+1} \gets \max\{(1-\alpha k_{n+1})\epsilon_{n+1},  0.5\epsilon_{n+1},\epsilon_{min}\} 
    $\;
    $
    STEP1=\max\left\{0.9k_n\left(\frac{tTOL}{tEST_1(n+1)}\right)^{1/2},0.5k_{n+1}\right\} 
    $\;
    $
    STEP2=\max\left\{0.9k_n\left(\frac{tTOL}{tEST_2(n+1)}\right)^{1/3},0.5k_{n+1}\right\} 
    $\;
    $
    k_{n+1} \gets \max\{STEP1,STEP2\} 
    $ \;
    REPEAT step \\
    }{
    \If {$EST_{n+1} < minTOL$ or $\min\{tEST_1(n+1),tEST_2(n+1)\} < min tTOL$}{
        $\epsilon_{n+2} \gets \min\{2\epsilon_{n+1},\epsilon_{max}\}$ \;
        $
        STEP1 \gets \max\left\{\min\left\{0.9k_{n+1}\left(\frac{tTOL}{tEST_1(n+1)}\right)^{1/2},2k_{n+1}\right\},0.5k_{n+1}\right\}
        $\;
        $
        STEP2 \gets \max\left\{\min\left\{0.9k_{n+1}\left(\frac{tTOL}{tEST_2(n+1)}\right)^{1/3},2k_{n+1}\right\},0.5k_{n+1}\right\}
        $ \;
        $
        k_{n+2} \gets \max\{STEP1,STEP2\}
        $\;
        CONTINUE \;
        }
    }
Pick method with larger time-step for next step:\\
\If {STEP1$>$STEP2}{
    $
    \bu_{n+1}=\bu_{n+1}^1
    $\\
}
Recover pressure $p_{n+1}$ if needed by: 
$
     p_{n+1}=-\frac{1}{\epsilon_{n+1}}\nabla\cdot\bu_{n+1}.
$
\caption{Variable $\epsilon$, variable time-step, variable order (VSVO) penalty method}
\end{algorithm}
\begin{remark}
In this algorithm, 0.9 is used as a standard safety factor \cite{layton2019doublyadaptive}.
\end{remark}
\section{Numerical Experiments}

\subsection{Modified Tayler-Green}
First we verify the adaptive $\epsilon$ penalty method does better than normal constant $\epsilon$ penalty method by comparing the adaptive $\epsilon$ tests (Algorithm 1) with two different constant $\epsilon$ options: 1) constant $\epsilon$ = $10^{-8}\nu$ and 2) constant $\epsilon = k$. Here option 1) is usually the approach using by engineering people and option 2) is derived from previous penalty paper by Shen \cite{Shen1995penalty}.

This test is also a modified version of the historically used problem Tayler-Green vortex \cite{victor2018time}. The exact solution is given by
\begin{align*}
    \bu(x,y,t)&=e^{-2\nu t}(\cos x\sin y,-\sin x\cos y),\\
    p(x,y,t)&=-\frac{1}{4}e^{-4\nu t}(\cos 2x+\cos 2y)+x(\sin 2t+\cos 3t)+y(\sin 3t+\cos 2t).
\end{align*}
This is inserted into the NSE and the body force $f(x,y,t)$ calculated.

The test was done using uniform meshes with 100 nodes per side of the square $[0, 2\pi]\times[0,2\pi]$. We solve by using $P_2$ elements and calculate up to time $T=25$. Here we still compare three different methods: 1) constant epsilon penalty method with $\epsilon=10^{-8}\nu$, 2) constant epsilon penalty method with $\epsilon=k$ and 3) variable penalty method (Algorithm 1). All three methods are calculated with $\Delta t=0.005$ with time-filter. The results are shown in \Cref{fig:quasi_divu_ut_plot}, \Cref{fig:quasi_divu_ut_zoomed_plot}, \Cref{fig:quasi_eps_uerr_plot} and \Cref{fig:quasi_perr_plot}.  

\Cref{fig:quasi_divu_ut_plot} shows the evolution of $\|\nabla\cdot u\|$ and discrete $\|u_t\|$. Constant penalty $\epsilon=k$ has much more larger $\|\nabla\cdot u\|$ and $\|u_t\|$ than both constant penalty $\epsilon=10^{-8}\nu$ and adaptive penalty methods. The results of constant $\epsilon=k$ is inaccurate as $\|\nabla\cdot u\|=\mathcal{O}(10^{-1})$.
\begin{figure}[h!]
    \centering
    \includegraphics[width=0.8\textwidth]{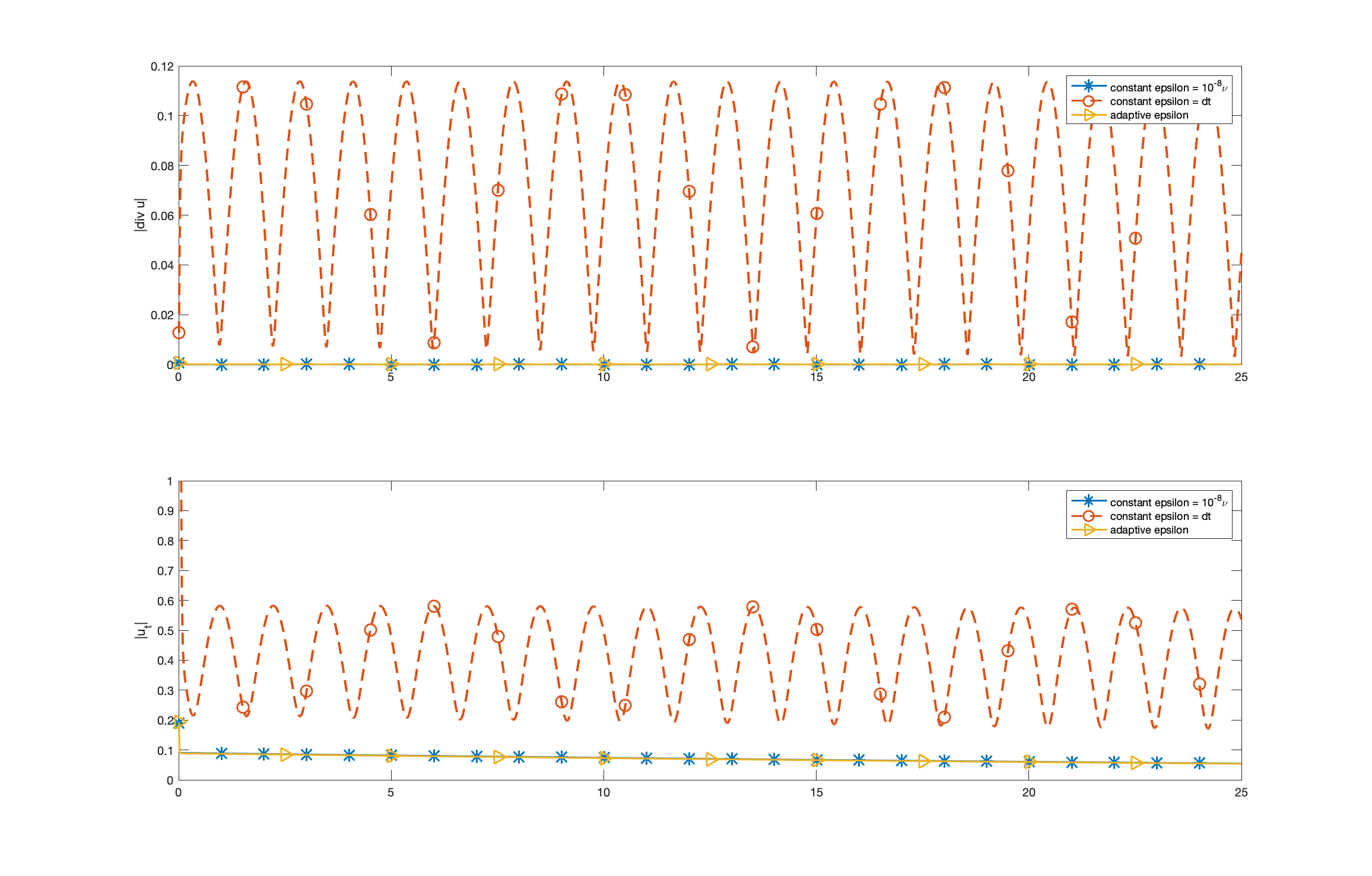}
    \caption{Comparison of $\|\nabla\cdot u\|$ and discrete $\|u_t\|$ between adaptive penalty (Algorithm 1) and two constant penalty methods, tests are done with 100 mesh points per side and $\Delta t=0.005$.}
    \label{fig:quasi_divu_ut_plot}
\end{figure}
To further see the difference between constant penalty $\epsilon=10^{-8}\nu$ and adaptive penalty, we zoomed in and got \Cref{fig:quasi_divu_ut_zoomed_plot}. Constant $\epsilon=10^{-8}\nu$ and adaptive penalty has comparable $\|u_t\|$ values, both of order $\mathcal{O}(10^{-2})$. But adaptive penalty has more oscillating $\|\nabla\cdot u\|$ values than constant penalty $\epsilon=10^{-8}\nu$. $\|\nabla\cdot u\|$ of adaptive penalty is $\mathcal{O}(10^{-5})$ and $\|\nabla\cdot u\|$ of constant penalty $\epsilon=10^{-8}\nu$ is $\mathcal{O}(10^{-9})$.
\begin{figure}[h!]
    \centering
    \includegraphics[width=0.8\textwidth]{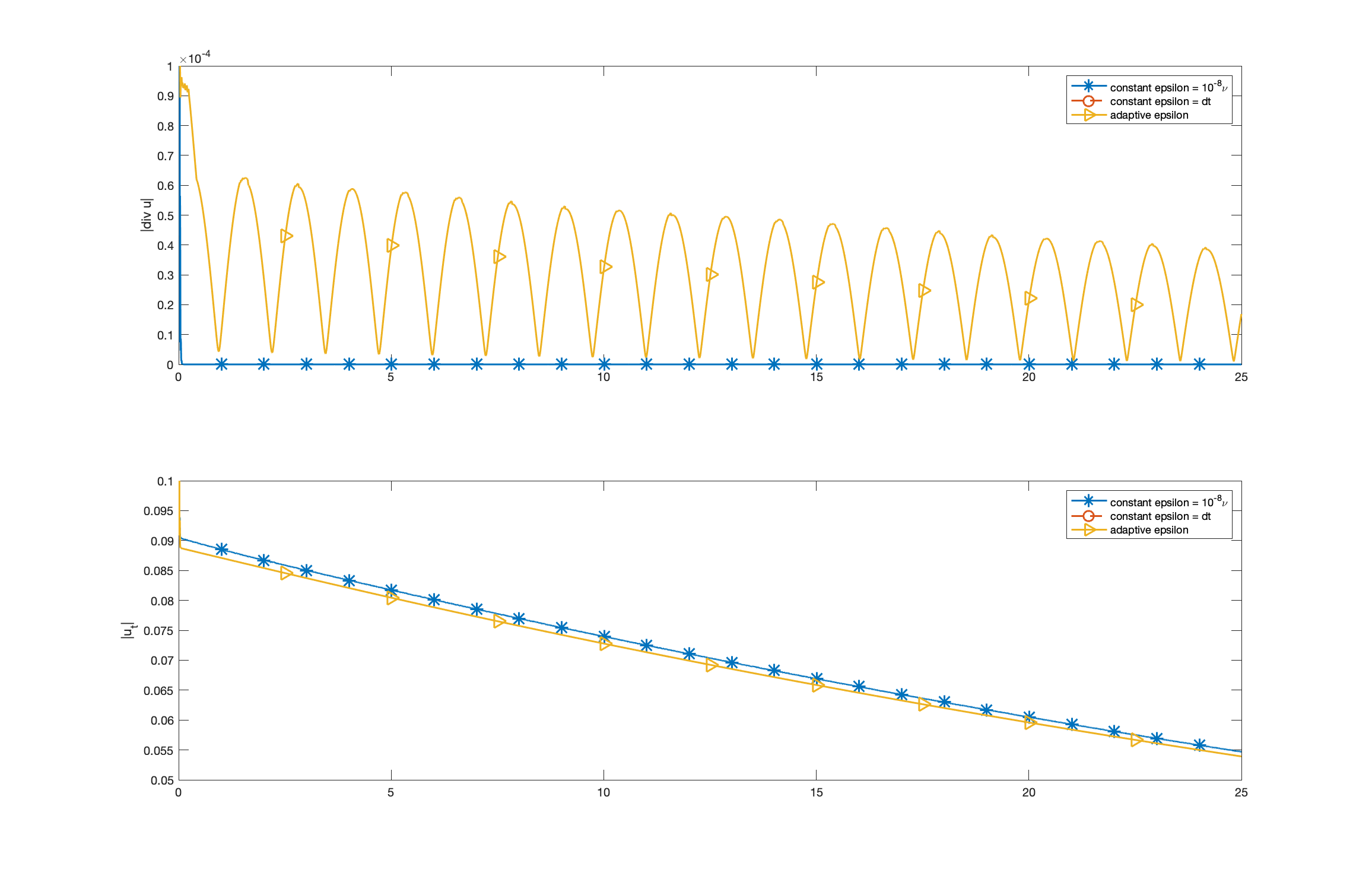}
    \caption{Zoomed in comparison of $\|\nabla\cdot u\|$ and discrete $\|u_t\|$ between adaptive penalty (Algorithm 1) and two constant penalty methods, tests are done with 100 mesh points per side and $\Delta t=0.005$.}
    \label{fig:quasi_divu_ut_zoomed_plot}
\end{figure}

The second plot of \Cref{fig:quasi_eps_uerr_plot} is the velocity error $\|u-u^h\|$, adaptive $\epsilon$ has a much more smaller error compared to the other two constant penalty method. The first plot of \Cref{fig:quasi_eps_uerr_plot} is the evolution of $\epsilon$ of 1) constant $\epsilon=10^{-8}\nu$ and 3) adaptive penalty. The evolution of $\epsilon$ of 2) constant $\epsilon=k$ is not shown in this plot due to the limitation of $y-$axis. $\epsilon$ of 3) adaptive penalty changes with time and gradually becomes stable over time. Because the penalty method is very sensitive to the choice of $\epsilon$ as we see in \Cref{fig:error_epsilon}. This shows that adaptive penalty method does pick a good $\epsilon$ automatically. Because penalty method is very sensitive to the choice of $\epsilon$ as we see in \Cref{fig:error_epsilon}, a good choice of $\epsilon$ is not easy at the beginning. Furthermore, by using the adaptive penalty method, we could eventually find that good $\epsilon$ with a little more calculation.
\begin{figure}[h!]
    \centering
    \includegraphics[width=0.8\textwidth]{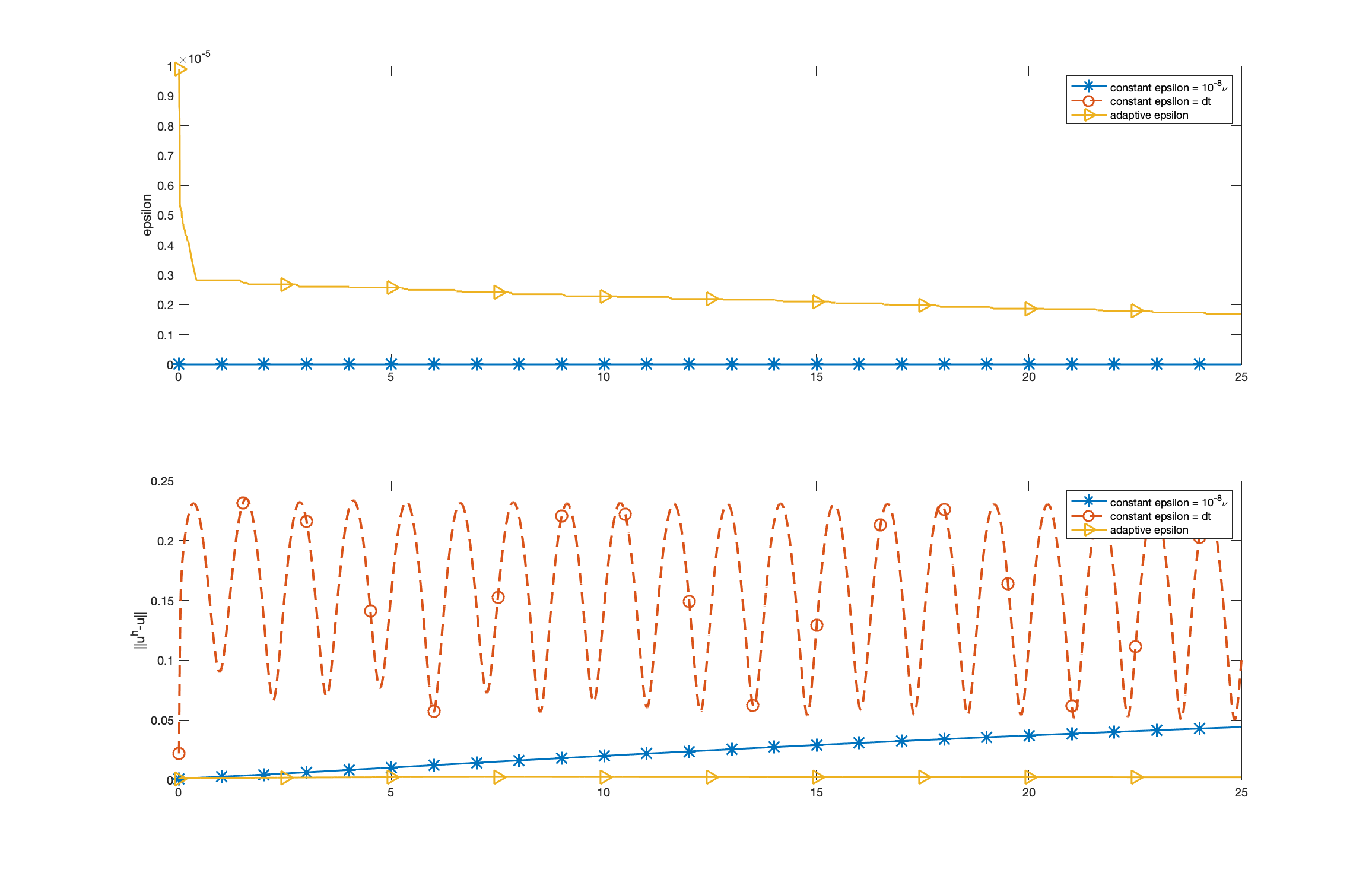}
    \caption{Evolution of $\epsilon$ and $\|u-u^h\|$ between adaptive penalty (Algorithm 1) and two constant penalty methods, tests are done with 100 mesh points per side and $\Delta t=0.005$.}
    \label{fig:quasi_eps_uerr_plot}
\end{figure}

The behavior of pressure is not good as seen in \Cref{fig:quasi_perr_plot}. Both pressure and pressure error fluctuate. Accurate pressure recovery remains an open question.
\begin{figure}[h!]
    \centering
    \includegraphics[width=0.8\textwidth]{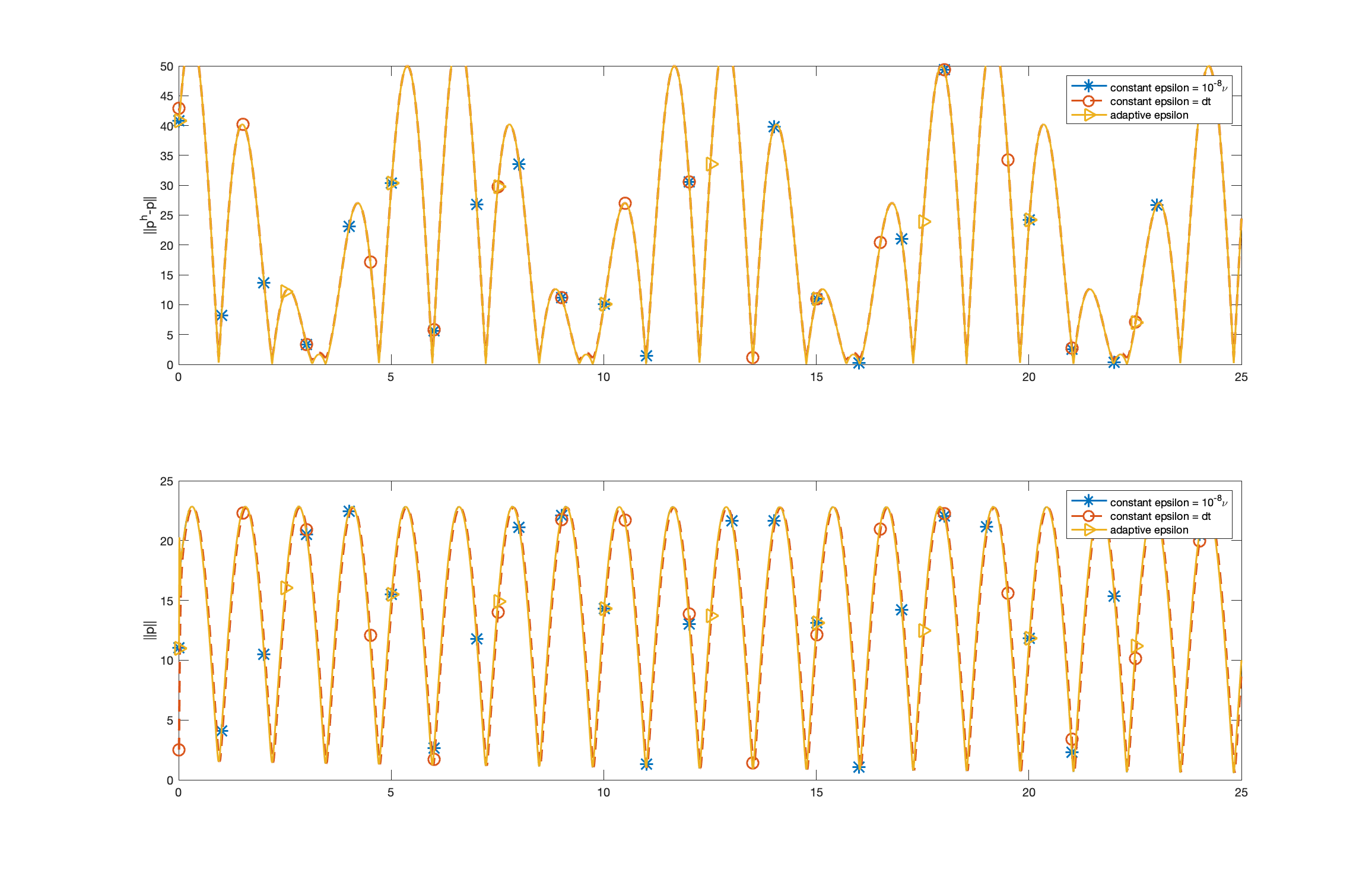}
    \caption{Comparison of $\|p-p^h\|$ and discrete $\|p^h\|$ between adaptive penalty (Algorithm 1) and two constant penalty methods, tests are done with 100 mesh points per side and $\Delta t=0.005$.}
    \label{fig:quasi_perr_plot}
\end{figure}

As a conclusion, there are three main advantages of the adaptive penalty method over the usual constant penalty method:
\begin{enumerate}
    \item The errors of adaptive $\epsilon$ and constant $\epsilon$ are comparable. Finding a good constant $\epsilon$ can require an exhaustive search.
    \item Constant $\epsilon=10^{-8}\nu$ behaviors better than constant $\epsilon=k$ in our tests. But as in the test, $\nu=0.01, \epsilon=10^{-8}\nu=10^{-10}$ leads to an extremely ill conditioned linear system. While adaptive $\epsilon$ levels out with $\epsilon\approx 10^{-5}$ and gives $\|\nabla\cdot u\|=\mathcal{O}(10^{-5})$. Adaptive $\epsilon$ controls $\|\nabla\cdot u\|$ better than $\epsilon=k$ and controls $\|\nabla\cdot u\|$ almost as well as $\epsilon=10^{-8}\nu$ but leads to a much better conditioned system. Further, adaptive $\epsilon$ has smaller velocity error than $\epsilon=10^{-8}\nu$. Overall, adaptive $\epsilon$ performed better.
    \item The only way to find the best $\epsilon$ is by exhaustive search for problems with already known solution. This is not possible for new problems but is not needed with the adaptive penalty.
\end{enumerate}

\subsection{A test with exact solution, taken from \cite{Victor2017artificial}}\label{exact-soln}
This exact solution experiment tests the accuracy of the adaptive penalty algorithm. 
The following test has exact solution for 2D Navier Stokes problem($\nu=1$). \\
Let the domain $\Omega=(-1,1)\times(-1,1)$. The exact solution is as follows:
\begin{align*}
    &\bu(x,y,t)=\pi\sin t(\sin 2\pi y\sin ^2 \pi x,-\sin 2\pi x\sin^2\pi y) \\
    &p(x,y,t)=\sin t\cos \pi x\sin\pi y
\end{align*}
This is inserted into the NSE and the body force $f(x,t)$ calculated. 

Uniform meshes were used with 270 nodes per side on the boundary. The mesh is fine enough that the error resulting from the meshsize is relatively smaller than that from the step-size. 
 Taylor-Hood elements (P2-P1) were used in this test. We ran the test up to $T=10.$ 
\begin{figure}[h!]
  \subfigure[$\Dt t=0.05$, minimum error occurs at $\epsilon=10^{-6}$]{
	\begin{minipage}[c][1\width]{
	   0.45\textwidth}
	   \centering
	   \includegraphics[width=1\textwidth]{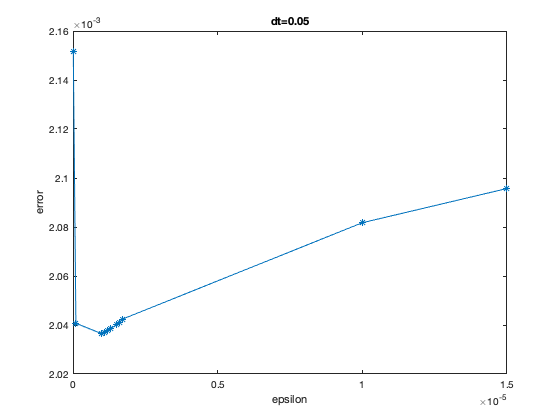}
	\end{minipage}
	}
 \hfill 	
  \subfigure[$\Dt t=0.02$, minimum error occurs at $\epsilon=5*10^{-9}$]{
	\begin{minipage}[c][1\width]{
	   0.45\textwidth}
	   \centering
	   \includegraphics[width=1\textwidth]{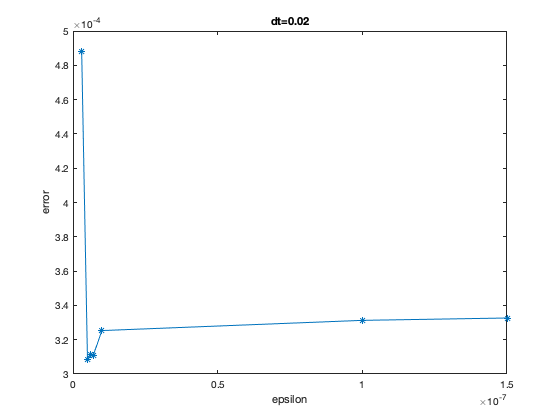}
	\end{minipage}
	}
\caption{$\|(\bu-\bu^h)(10)\|$ with constant time-step and different values of penalty parameter $\epsilon$} 
\label{fig:error_epsilon}
\end{figure}

\Cref{fig:error_epsilon} indicates that velocity error is very sensitive to the choice of $\epsilon$. This sensitivity is a known effect, motivating the $\epsilon-$adaptive algorithm of the penalized NSE \eqref{pnse_eliminate}. It also suggests that $\epsilon$ too large is safer than $\epsilon$ too small.

\subsubsection{Test 1: Constant time-step, variable $\epsilon$ test}
First, we tested the constant time-step, variable $\epsilon$ test based on Algorithm 1. The error at final time T=10 is at \Cref{constant_error}. We observe that the velocity error are good, but the pressure approximation is poor. 
\begin{table}[H]
    \centering
    \begin{tabular}{||c|c|c|c|c|c|c|c||}
    \hline
     dt   &\# steps & $\|(u-u^h)(10)\|$ &rate& $\|(u-u^h)(10)\|_{L^\infty}$ &rate& $\|(p-p^h)(10)\|_{L^\infty}$&rate \\
     \hline
     0.1  &100  & 0.00965699&- &0.00869&- &0.268895&- \\
     \hline
     0.05  &200& 0.00203366 &2.2475& 0.002075&2.0662&0.229891&0.2261 \\
     \hline
     0.02 &500& 0.000332169 &1.9775& 0.0004969&1.5599&0.222597&0.0352\\
     \hline
     0.01 &1000& 0.000324625 &0.0331&0.00043955 &0.1769&0.196683&0.1786 \\
     \hline
    \end{tabular}
    \caption{Constant time-step variable $\epsilon$ error comparison}
    \label{constant_error}
\end{table}
\begin{figure}[h!]
    \centering
    \includegraphics[width=\textwidth,height=9cm]{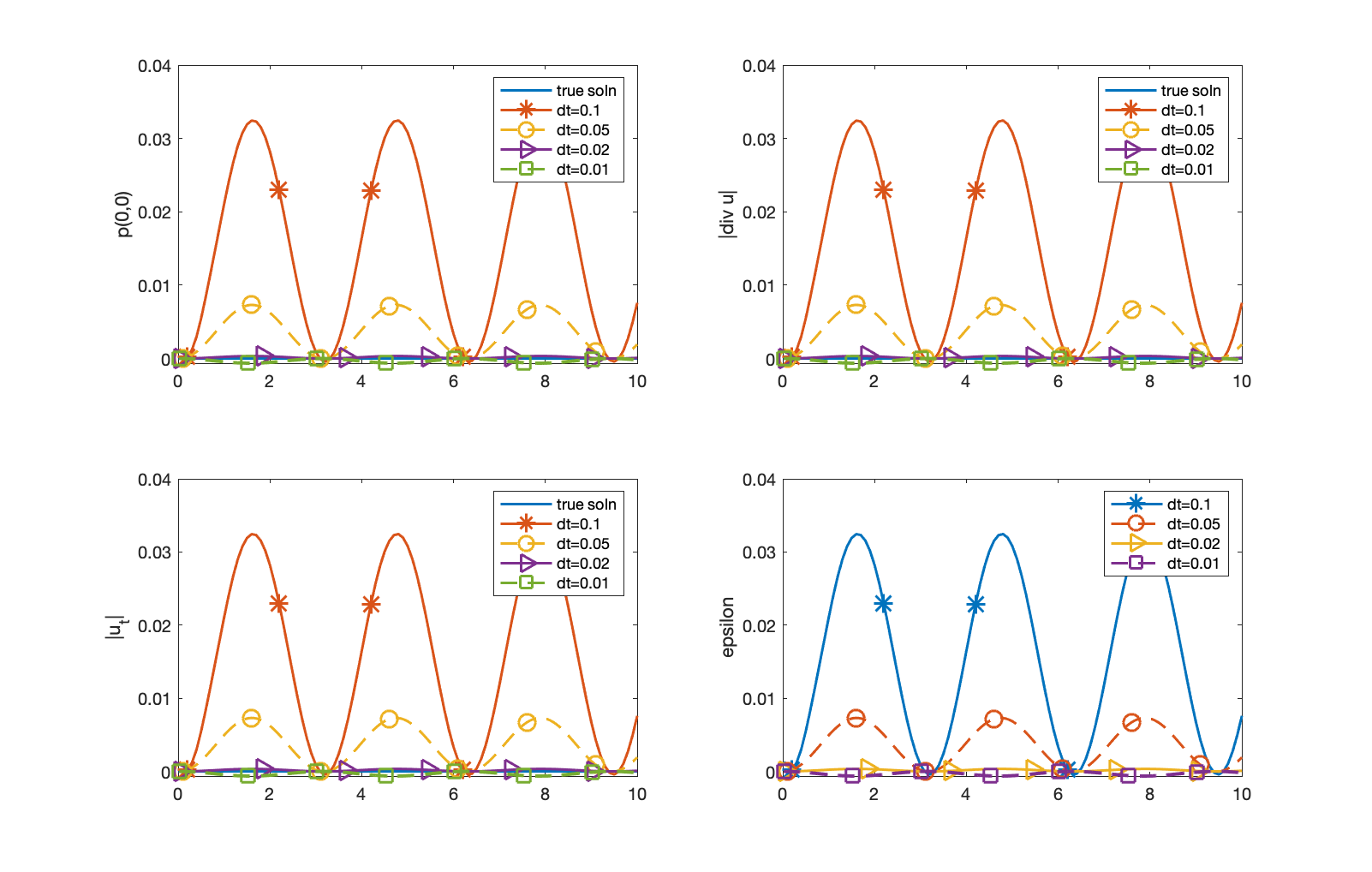}
    \caption{Comparison of results with different $\Dt t$ of variable $\epsilon$, constant time-step method (Algorithm 1)}
    \label{fig:compare of alg1}
\end{figure}
$\|\nabla\cdot\bu\|$ in \Cref{fig:compare of alg1} is well controlled for all three different time-steps. $\|\bu_t\|$ in \Cref{fig:compare of alg1} is very close to the true value of problem. Decreasing the time-step improves accuracy of the pressure at origin $(0,0)$. The oscillations in the errors of $p(0,0)$ and $\|\nabla\cdot\bu\|$ arise from the multiplier $\sin(t)$ in the exact solution.

This test has a smooth solution, $\|\nabla\bu\|$ does not vary too much in the whole test. And this result in the estimator $EST=\|\nabla\cdot\bu\|/\|\nabla\bu\|$ is also very smooth. So $\epsilon$ does not vary too much in this test
. Both $\|\bu_t\|$ and $\|\nabla\cdot\bu\|$ are well controlled. The values of $\|\bu_t\|$ are very close to the true value. The values of $\|\nabla\cdot\bu\|$ are very close to 0 (up to $10^{-4}$) which is the incompressibility condition. One surprising effect we see from the plot of $\|\nabla\cdot\bu\|$ is that smaller
 $\Dt t$ leads to larger $\|\nabla\cdot\bu\|$, which contradicts expectations from theory. This is due to for smaller $\Dt t$, this adaptive $\epsilon$ algorithm suggests larger $\epsilon$. $\nabla\cdot\bu$ and $p$ need to satisfy the relation $\nabla\cdot\bu+\epsilon p=0$ and this implies $\|\nabla\cdot\bu\|=\epsilon\|p\|$. 
\begin{figure}[h!]
    \centering
    \includegraphics[width=0.6\textwidth]{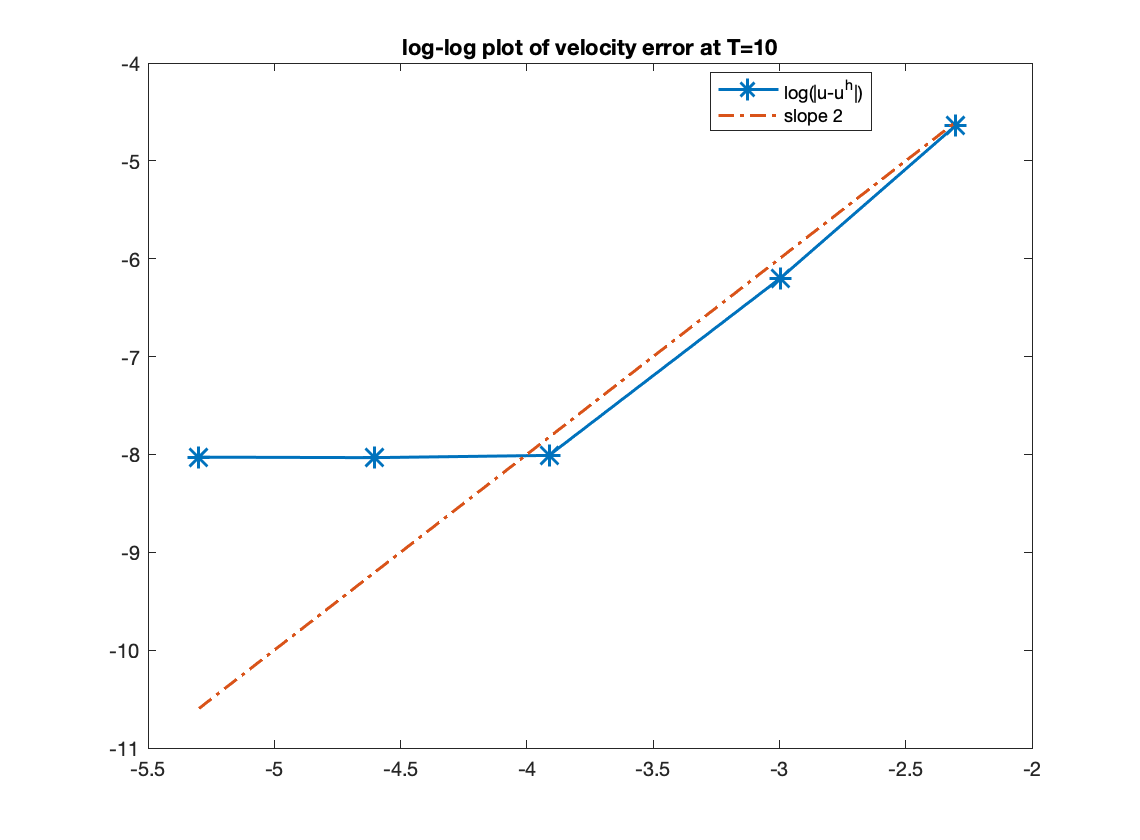}
    \caption{log-log plot of velocity error at final step T=10 $\|(\bu-\bu^h)(10)\|$ v.s. time-step using variable $\epsilon$ constant time-step method (Algorithm 1). Slope of plot $\|(u-u^h)(10)\|$ is close to 2.}
    \label{fig:log-log}
\end{figure}

\Cref{fig:log-log} is the log-log plot of velocity error at final time T=10 versus the time-step k. We see the curve of $\log(\|(u-u^h)(10)\|)-\log(k)$ has slope close to 2. This constant time-step, variable $\epsilon$, backward Euler algorithm with time filter (Algorithm 1) is second-order accurate. When the time-step gets too small, the error does not change too much. This is due to the choice of tolerance, TOL, for algorithm here is $10^{-6}$ and at this time-step reached the error plateau.

\subsubsection{Test 2: Double Adaptive}
Next, we test the same problem using the variable time-step algorithm (Algorithm 2-4). The errors of variable time-step, variable $\epsilon$ method are presented at \Cref{variable_error}. From the table, the variable order method gives slightly better results than first-order and second-order methods. The velocity error is of order  $10^{-3}$ using VSVO and is of order $10^{-2}$ for both first and second order algorithm. The pressure error of VSVO is approximately  50\% smaller than first and second order algorithms. 
\begin{table}[h!]
    \centering
    \begin{tabular}{||c|c|c|c|c||}
    \hline
      method   & $\#$ steps& $\|(u-u^h)(10)\|$ & $\|(u-u^h)(10)\|_{L^\infty}$ & $\|(p-p^h)(10)\|_{L^\infty}$ \\
      \hline
       first &3450&0.0609278&0.0512269 &0.348461 \\
       \hline
       second & 447 &0.0567343 &0.0476828 &0.344317 \\
       \hline
       vsvo &566&0.00364638 &0.00314834 &0.190205 \\
       \hline
    \end{tabular}
    \caption{Variable time-step error comparison}
    \label{variable_error}
\end{table}
\begin{figure}[H]
    \centering
    \includegraphics[width=\textwidth,height=10cm]{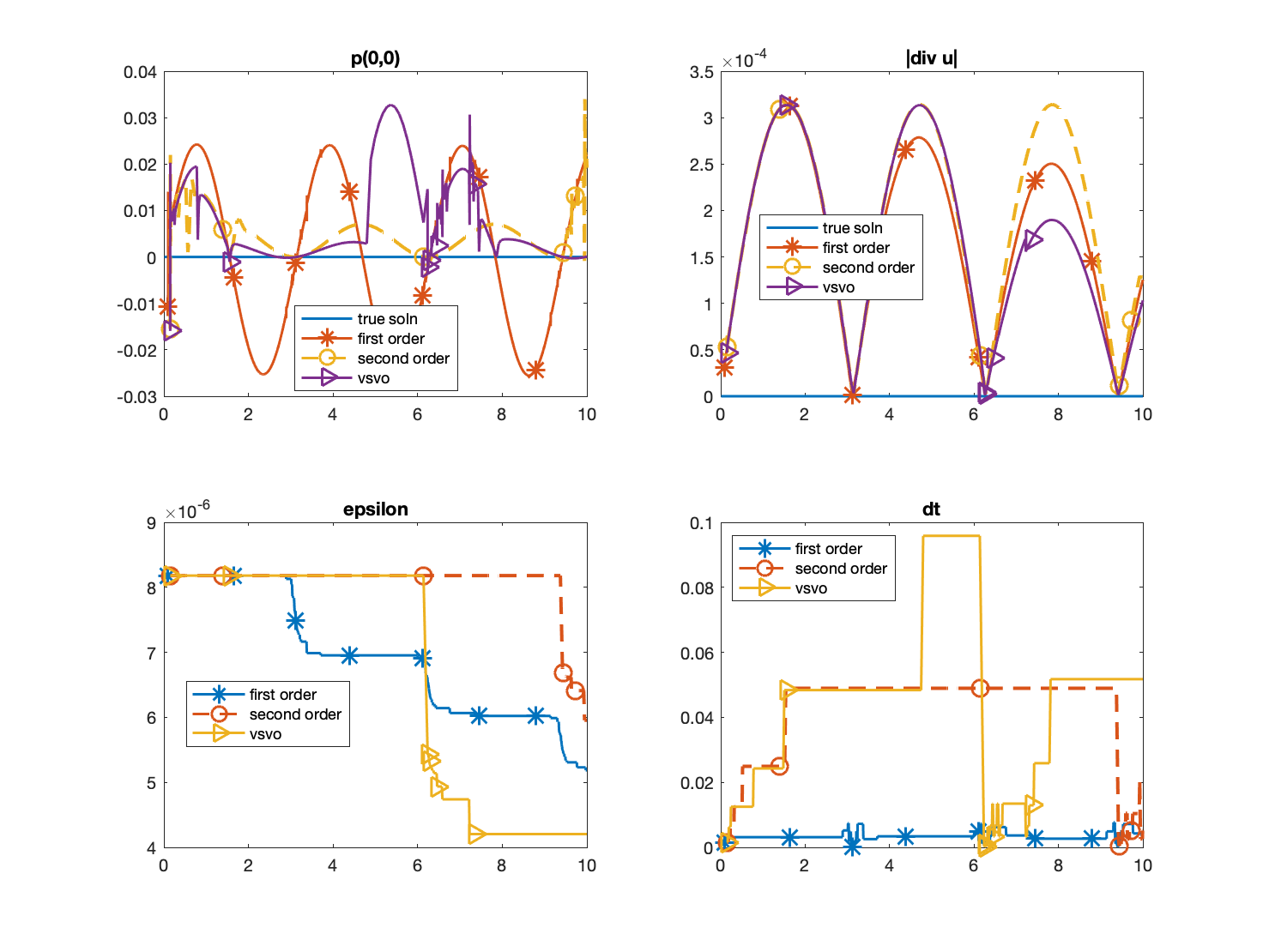}
    \caption{Comparison of variable time-step, variable $\epsilon$ method (Algorithm 2,3,4)}
    \label{fig:variable-step}
\end{figure}
For the variable time-step methods, in  \Cref{fig:variable-step} we track the evolution of $\epsilon$ and $\Dt t$, the pressure at the origin and $\|\nabla\cdot u\|$. The last plot of  \Cref{fig:variable-step} shows that the second-order method consistently chooses a larger time-step than the first-order method. At the beginning, VSVO picks the second-order method, and after some time VSVO picks the first-order method. The VSVO algorithm takes larger time-steps than both the first and second order method. 
\begin{figure}[h!]
    \centering
    \includegraphics[width=0.7\textwidth]{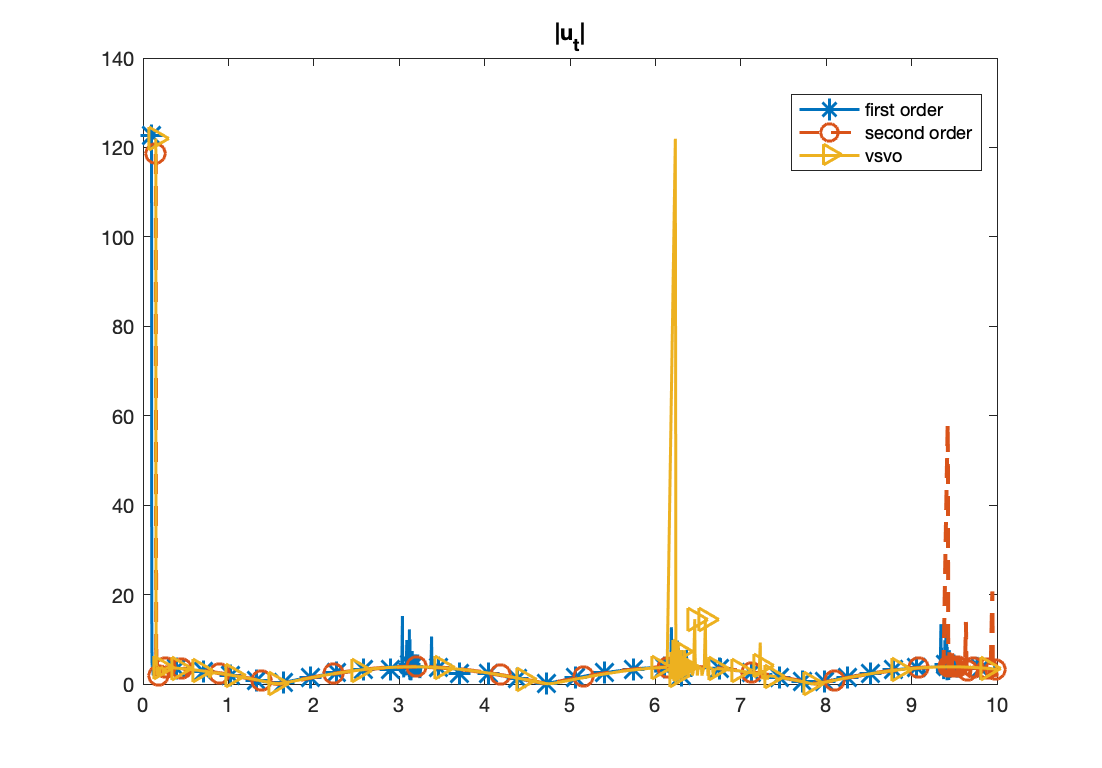}
    \caption{$\|u_t\|$ plot of variable time-step, variable $\epsilon$ method (Algorithm 2,3,4) The spikes show at the time when $\epsilon$ decrease too fast (violation of  \eqref{epsilon-restriction})}
    \label{fig:ut_spike}
\end{figure}
From \Cref{fig:ut_spike}, the plot of $\|\bu_t\|$, we see some spikes. The time where we see spikes is exactly $\epsilon$ is decreased a lot (see \Cref{fig:variable-step}.) From the analysis of stability of $\|\bu_t\|$, when we decrease $\epsilon$, \eqref{epsilon-restriction} must be satisfied to avoid catastrophic growth of $\|\bu_t\|$. When adapting $\epsilon$ and $k$ to ensure $EST<TOL$, we may reject due to $EST$ exceeding $TOL$ and redo the step several times. This may result in the sudden decrease of $\epsilon$, as in \Cref{fig:variable-step}. This kind of sudden decrease of $\epsilon$ violates \eqref{epsilon-restriction} and results in the spikes as in \Cref{fig:ut_spike}. This illustrates the necessity of controlling the change in $\epsilon$ using the method we derived from stability analysis and decreasing $\epsilon$ using \eqref{epsilon-restriction}: $(1-k\alpha)\epsilon_n \leq \epsilon_{n+1}$.

\subsection{Flow Between Offset Circles, taken from \cite{layton2019doublyadaptive}}\label{fbc}
The domain is a disk with a smaller off center obstacle inside. Let $r_1=1,r_2=0.1,c=(c_1,c_2)=(1/2,0),$ then the domain is given by
\beas
\Omega=\{(x,y):x^2+y^2< r_1^2 \;\text{and}\; (x-c_1)^2+(y-c_2)^2> r_2^2\}.
\eeas
The flow is driven by a counterclockwise rotational body force
\beas
f(x,y,t)=\min\{t,1\}(-4y*(1-x^2-y^2),4x*(1-x^2-y^2))^T,\;\;\;\text{for}\; 0\leq t\leq 10,
\eeas
with no-slip boundary conditions on both circles. We discretize in space using $P^2-P^1$ Taylor-Hood elements. There are 200 mesh points around the outer circle and 50 mesh points around the inner circle. The finite element discretization has a maximal mesh width of $h_{max}=0.048686$. The flow is driven by a counterclockwise force (f=0 on the outer circle). The flow rotates about the origin and interacts with the immersed circle.  

To better compare the results, tests is also done using the following algorithm: \\
Backward Euler with grad-div stabilization parameter $\gamma=1$ see Jenkins, John, Linke and Rebholz \cite{graddiv-jenkins2014parameter}
\begin{equation}\label{grad-div-stab}
\begin{aligned}
    \frac{\bu^{n+1}-\bu^n}{k}+\bu^n\cdot\nabla\bu^{n+1}+\frac{1}{2}(\nabla\cdot\bu^n)\bu^{n+1}+\nabla p^{n+1}-\nu\Delta\bu^{n+1}-\gamma\nabla\nabla\cdot\bu^{n+1}=f_{n+1} ,\\
    \nabla\cdot\bu^{n+1}=0 .
\end{aligned}
\end{equation}
\begin{figure}[h!]
    \centering
    \includegraphics[width=\textwidth,height=13cm]{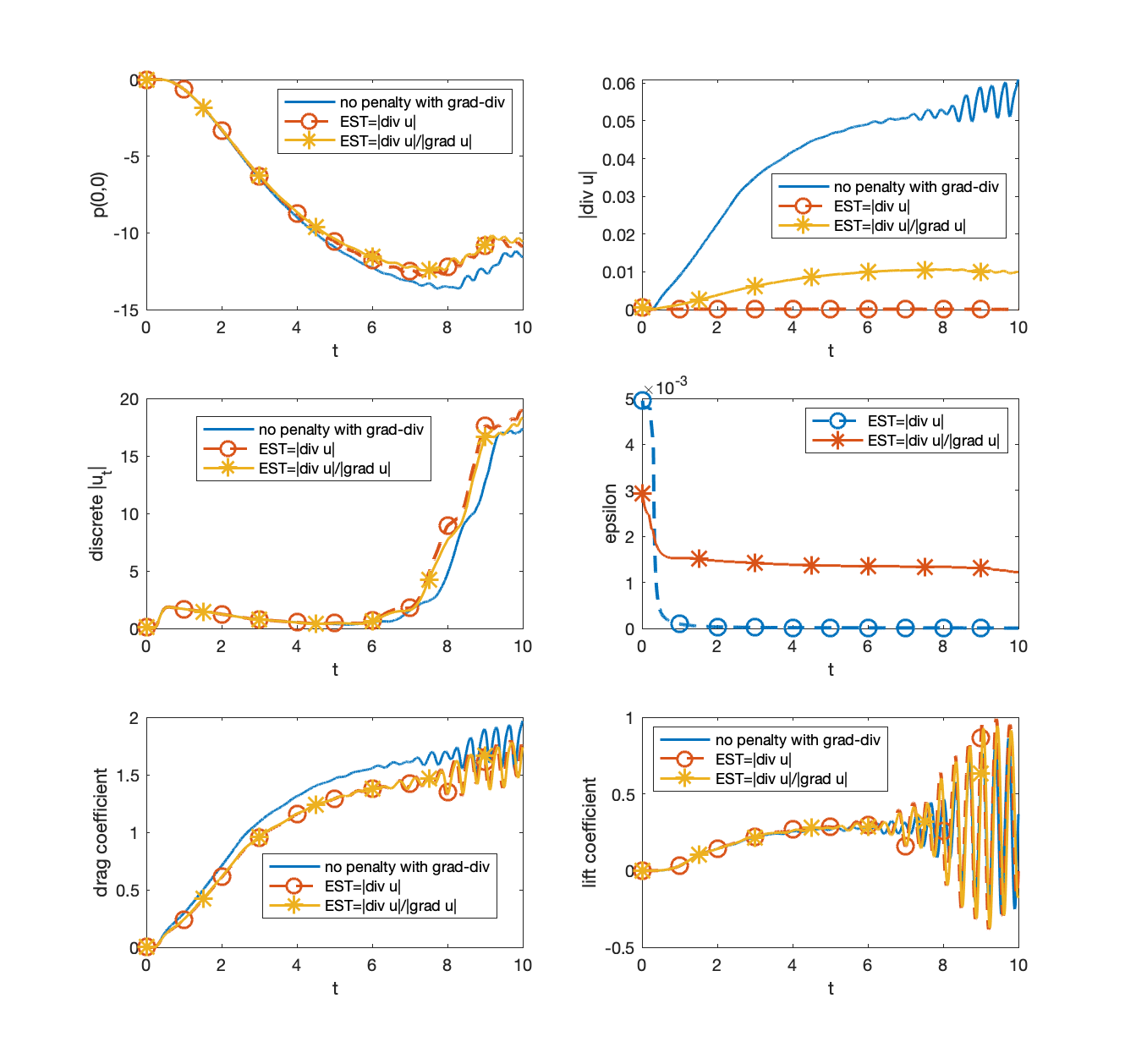}
    \caption{Comparison between different estimators $\|\nabla\cdot\bu\|$ and $\|\nabla\cdot\bu\|/\|\nabla \bu\|, Re=100, \Dt t=0.005$ with Algorithm 1 (constant time-step, variable $\epsilon$). No penalty uses Backward Euler with grad-div stabilization \eqref{grad-div-stab} with $\Dt t=0.001$. Tests without penalty use 320 mesh points around the outer circle and 80 mesh points around the inner circle. The finite element discretization has a maximal mesh width of $h_{max}=0.0347224$.}
    \label{fig:flow-test}
\end{figure}
For both estimators $EST=\|\nabla\cdot \bu\|$ and $EST=\|\nabla\cdot \bu\|/\|\nabla \bu\|$, we use constant time-step variable $\epsilon$ (Algorithm 1) with the same tolerance $TOL$ and lower tolerance $minTOL$. We track the evolution of $\epsilon$, the pressure at the origin, the evolution of $\|\nabla\cdot u\|,$ $\|u_t\|$ and the lift, drag coefficients. These are all shown in \Cref{fig:flow-test}. 

The fourth plot in \Cref{fig:flow-test} shows that with $EST=\|\nabla\cdot u\|/\|\nabla u\|$ chooses larger $\epsilon$ values than with estimator $EST=\|\nabla\cdot u\|$. The evolution of pressure, lift and drag coefficients behave similarly for both estimators.
$\|\nabla\cdot\bu\|$ is smaller using adaptive $\epsilon$ penalty algorithm (Algorithm 1) than using Backward Euler with grad-div stabilization \eqref{grad-div-stab}. $\|\nabla\cdot\bu\|$ from the penalty method is at least 5 times smaller $\|\nabla\cdot\bu\|$ from coupled Backward Euler with grad-div stabilization \eqref{grad-div-stab}. The adaptive penalty method has better control of $\|\nabla\cdot\bu\|$ than the coupled Backward algorithm with grad-div stabilization term \eqref{grad-div-stab}. Lift coefficient calculated from adaptive $\epsilon$ penalty method looks good. 

\section{Conclusions and open problems} This paper presents a stability and error analysis for the adaptive $\epsilon$ penalty method. Also, four different algorithms for both constant and variable time-step were introduced. There remain open problems and algorithmic improvements possible in the future. In this paper, we introduced the adaptive $\epsilon$ scheme with a condition from stability analysis which could ensure the stability of the result. It is unclear how sharp this bound is, or if the restriction \eqref{epsilon-restriction} is necessary in all time-steps. Further, by rejecting and repeating steps to guarantee EST<TOL result in violating the restriction \eqref{epsilon-restriction}. For different time-step, the problem has a different optimal $\epsilon$ value. An algorithm that adapts $\epsilon$ and k independently may be inferior to one that relates the step size to the penalty parameter. However, it there is not an obvious relation between $\epsilon$ and k, so further research may be necessary to find a more efficient doubly adaptive algorithm.
The pressure recovered directly from the continuity equation, $\nabla\cdot\bu+\epsilon p=0$ \eqref{pnse1} is not good estimate to the pressure from coupled system. We can look into alternate ways to recover the pressure, such as using the Pressure Poisson equation (PPE) see Kean and Schneier  \cite{Kean2019error}.

\section*{Acknowledgement}
We would like to thank Professor William Layton for suggesting the problem and for his help throughout the research.
\bibliographystyle{plain}
\bibliography{main}

\Appendix

\section{Detailed Algorithms}
\begin{algorithm}[h]
\SetAlgoLined
Given $\bu_n,\bu_{n-1},\epsilon_{n+1},\epsilon_n$, tolerance for $\epsilon$: TOL=$10^{-6}$ and lower tolerance minTOL=TOL/10, lower and upper bound of $\epsilon: \epsilon_{min}=10^{-8},\epsilon_{max}=10^{-5}, \alpha=2$,tolerance for $\Delta t$: tTOL=$10^{-5}$ and lower tolerance mintTOL=tTOL/10  \\
\textit{Compute} $\tau=\frac{k_{n+1}}{k_n}$ and $\alpha_1=\frac{\tau(1+\tau)}{1+2\tau}$\\
Solve for $\bu_{n+1}^1$ \\
Set $\bu^\star=(1+\tau)\bu_n-\tau\bu_{n-1}$
\begin{align*}
    \frac{\bu_{n+1}^1-\bu_n}{k_{n+1}}+\bu^*\cdot\nabla\bu_{n+1}^1+\frac{1}{2}(\nabla\cdot\bu^*)\bu_{n+1}^1-\nabla\left(\frac{1}{\epsilon_{n+1}}\nabla\cdot\bu_{n+1}^1\right)-\nu\Delta\bu_{n+1}^1=f_{n+1}.
\end{align*}
Compute estimator EST and difference $D_2$
\begin{align*}
    D_2(n+1)&=\frac{2k_n}{k_n+k_{n+1}}\bu_{n+1}^1-2\bu_n+\frac{2k_{k+1}}{k_n+k_{n+1}}\bu_{n-1} ,\\
    EST_e(n+1)&=\|\nabla\cdot \bu_{n+1}\|/\|\nabla\bu_{n+1}\|, \\
    tEST_1(n+1)&=\frac{\alpha_1}{2}\|D_2(n+1)\|.
\end{align*}
Adapt $\epsilon$ and $k$ using the standard decision tree:\\
\eIf {$EST_e(n+1) > TOL$ or $tEST_1(n+1) > tTOL$} {
    $
    \epsilon_{n+1} \gets \max\{(1-\alpha k_{n+1})\epsilon_{n+1}, \; 0.5\epsilon_{n+1}, \epsilon_{min}\} 
    $\;
    $
    k_{n+1} \gets \max\left\{0.9k_n\left(\frac{tTOL}{tEST_1(n+1)}\right)^{1/2},0.5k_{n+1}\right\} 
    $ \;
    REPEAT step \\
    }{
    \If {$EST_{n+1} < minTOL$ or $tEST_{n+1} < min tTOL$}{
        $\epsilon_{n+2} \gets \min\{2\epsilon_{n+1},\epsilon_{max}\}$ \;
        $
        k_{n+2} \gets \max\left\{\min\left\{0.9k_{n+1}\left(\frac{tTOL}{tEST_1(n+1)}\right)^{1/2},2k_{n+1}\right\},0.5k_{n+1}\right\}
        $ \;
        CONTINUE \;
        }
    }
Recover pressure $p_{n+1}$ if needed by:
$
p_{n+1}=-\frac{1}{\epsilon_{n+1}}\nabla\cdot\bu_{n+1}.
$
\caption{Variable $\epsilon$, variable time-step, first-order penalty method}
\end{algorithm}

\begin{algorithm}[h]
\SetAlgoLined
Given $\bu_n,\bu_{n-1},\epsilon_{n+1},\epsilon_n$, tolerance for $\epsilon$ TOL=$10^{-6}$ and lower tolerance minTOL=TOL/10, lower and upper bound of $\epsilon: \epsilon_{min}=10^{-8},\epsilon_{max}=10^{-5}, \alpha=2$,tolerance for $\Delta t$: tTOL=$10^{-5}$ and lower tolerance mintTOL=tTOL/10  \\
\textit{Compute} $\tau=\frac{k_{n+1}}{k_n}$ and $\alpha_1=\frac{\tau(1+\tau)}{1+2\tau},\alpha_2=\frac{\tau_n(\tau_{n+1}\tau_n+\tau_n+1)(4\tau_{n+1}^3+5\tau_{n+1}^2+\tau_{n+1})}{3(\tau_n\tau_{n+1}^2+4\tau_n\tau_{n+1}+2\tau_{n+1}+\tau_n+1)}$ \\
Set $\bu^*=(1+\tau)\bu_n-\tau\bu_{n-1}$ \\
Solve for $\bu_{n+1}^1$ \\
\begin{align*}
    \frac{\bu_{n+1}^1-\bu_n}{k_{n+1}}+\bu^*\cdot\nabla\bu_{n+1}^1+\frac{1}{2}(\nabla\cdot\bu^*)\bu_{n+1}^1-\nabla\left(\frac{1}{\epsilon_{n+1}}\nabla\cdot\bu_{n+1}^1\right)-\nu\Delta\bu_{n+1}^1=f_{n+1}.
\end{align*}
Compute estimator EST and difference $D_2$ and apply time filter
\begin{align*}
    D_2(n+1)&=\frac{2k_n}{k_n+k_{n+1}}\bu_{n+1}^1-2\bu_n+\frac{2k_{k+1}}{k_n+k_{n+1}}\bu_{n-1}, \\
    \bu_{n+1}&=\bu_{n+1}^1-\frac{\alpha_1}{2}D_2(n+1), \\
    EST_e(n+1)&=\|\nabla\cdot \bu_{n+1}\|/\|\nabla\bu_{n+1}\| ,\\
    tEST_2(n+1)&=\frac{\alpha_2}{6}\left\|\frac{3k_{n-1}}{k_{n+1}+k_n+k_{n-1}}D_2(n+1)-\frac{3k_{n+1}}{k_{n+1}+k_n+k_{n-1}}D_2(n)\right\|.
\end{align*}
\eIf {$EST_e(n+1) > TOL$ or $tEST_2(n+1) > tTOL$} {
    $
    \epsilon_{n+1} \gets \max\{(1-\alpha k_{n+1})\epsilon_{n+1},  0.5\epsilon_{n+1}, \epsilon_{min}\} 
    $\;
    $
    k_{n+1} \gets \max\left\{0.9k_n\left(\frac{tTOL}{tEST_1(n+1)}\right)^{1/3},0.5k_{n+1}\right\} 
    $ \;
    REPEAT step \\
    }{
    \If {$EST_{n+1} < minTOL$ or $tEST_{n+1} < min tTOL$}{
        $\epsilon_{n+2} \gets \min\{2\epsilon_{n+1},\epsilon_{max}\}$ \;
        $
        k_{n+2} \gets \max\left\{\min\left\{0.9k_{n+1}\left(\frac{tTOL}{tEST_1(n+1)}\right)^{1/3},2k_{n+1}\right\},0.5k_{n+1}\right\}
        $ \;
        CONTINUE \;
        }
    }
Recover pressure $p_{n+1}$ by: 
$
     p_{n+1}=-\frac{1}{\epsilon_{n+1}}\nabla\cdot\bu_{n+1}.
$
\caption{Variable $\epsilon$, variable time-step, second-order penalty method}
\end{algorithm}
\end{document}